\documentclass[a4paper,leqno]{amsart}

\usepackage{amsmath}
\usepackage{amsthm}
\usepackage{amssymb}
\usepackage{amsfonts}
\usepackage[applemac]{inputenc}
\usepackage{mathrsfs}

\usepackage{pdfsync}
\usepackage{color}

\usepackage{mathabx}

\def\C{{\mathbb C}}% complex numbers
\def\R{{\mathbb R}}% real numbers
\def\N{{\mathbb N}}% nonnegative integers
\def\ge{\geqslant}%greaterorequal

\newcommand{\eps}{\varepsilon}

\theoremstyle{plain}
\newtheorem{theorem}{Theorem}[section]
\newtheorem{lemma}[theorem]{Lemma}
\newtheorem{corollary}[theorem]{Corollary}
\newtheorem{proposition}[theorem]{Proposition}

\theoremstyle{definition}

\newtheorem{remark}[theorem]{Remark}
\newtheorem*{remark*}{Remark}

\numberwithin{equation}{section}

\begin{document}

\title[Strong magnetic field limit in a nonlinear Iwatsuka-type model]
{Strong magnetic field limit in a nonlinear Iwatsuka-type model}

\author[E. Richman]{Evelyn Richman}

\address[E.~Richman]
{Department of Mathematics, Statistics, and Computer Science, M/C 249, University of Illinois at Chicago, 851 S. Morgan Street, Chicago, IL 60607, USA}
\email{jrichm2@uic.edu}

\author[C. Sparber]{Christof Sparber}

\address[C.~Sparber]
{Department of Mathematics, Statistics, and Computer Science, M/C 249, University of Illinois at Chicago, 851 S. Morgan Street, Chicago, IL 60607, USA}
\email{sparber@math.uic.edu}

\begin{abstract}
We study the strong magnetic field limit for a nonlinear Iwatsuka-type model, i.e. a nonlinear Schr\"odinger equation in two 
spatial dimensions with a magnetic vector potential that only depends on the $x$-coordinate. 
Using a high-frequency averaging technique, we show that this equation can be effectively described by a nonlocal nonlinear model, which is no longer dispersive. 
We also prove that, in this asymptotic regime, inhomogeneous nonlinearities are confined along the $y$-axis. 
\end{abstract}

\date{\today}


\thanks{This publication is based on work supported by the NSF through grant no. DMS-1348092}
\maketitle

%%%%%%%%%%%%%%%%%%%%%%%%%%%%%%%%%%%
%%%%%%%%%%%%%%%%%%%%%%%%%%%%%%%%%%%

\section{\textbf{Introduction}}\label{sec:intro}

The spectral and dynamical properties of magnetic Schr\"odinger equations have been of interest for many years, see, e.g., \cite{Erd} and the references therein. 
A particular application of such equations arises in the mathematical description of 
Quantum Hall systems. The typical Hall system is modeled by an electron moving in the plane $\R^2$ subject to a constant transverse magnetic field, cf. \cite{HiSo}. The corresponding 
Schr\"odinger operator is 
\begin{equation}\label{H}
\mathcal H =\frac{1}{2} \big(-i \nabla +A(x_1, x_2)\big)^2\quad \text{on $L^2(\R^2)$,}
\end{equation}
and the corresponding magnetic field is given by the third component of the cross product $B(x_1,x_2) =  (\nabla \times A(x_1, x_2))_3$. 
In \cite{Iwa1, Iwa2}, Iwatsuka considered the situation where the magnetic field is a function of $x_1\in \R$ only, i.e. $B(x_1, x_2) = B(x_1)$. Using a particular choice of gauge, 
this can be achieved via a vector potential of the form $$A(x_1,x_2)= (0, \beta(x_1)).$$

The aim of this paper is to describe how strong fields of this particular type asymptotically yield a purely magnetic confinement within the two-dimensional quantum dynamics. 
In the linear case such questions have recently been addressed in \cite{CHSV}. 
Here, we are mainly interested in the situation where additional nonlinear self-interactions are present. Such magnetic nonlinear Schr\"odinger equations describe, e.g., the dynamics of 
Fermion pairs in the low-density limit of BCS theory \cite{HS}, and we also refer to \cite{Mi, naka} for recent mathematical results concerning the Cauchy problem of magnetic 
NLS.  
In the following, we shall study the case where 
\begin{equation}\label{B}
\beta(x_1) = \frac{b}{\eps^2} x_1,\quad b \in \R,
\end{equation}
with $0<\eps\ll 1$ a small, dimensionless parameter. In the limit $\eps\to 0_+$, this yields a large (constant) magnetic field with strength $|B|=\mathcal O(\frac{1}{\eps^2})$. 
The nonlinear Iwatsuka-type model we shall consider is then given by
\begin{equation}\begin{split}\label{pde1}
		i\partial_t \psi =&\,\frac{1}{2}\big(-i\nabla + A^\varepsilon(x_1)\big)^2\psi + \varepsilon^{\sigma}\lambda |\psi|^{2\sigma}\psi\\
		=&\, \frac{1}{2}\left(-\partial_{{x_1}}^2 -\partial_{{x_2}}^2 - \frac{i2b}{\varepsilon^2}{x_1}\partial_{{x_2}} + \frac{1}{\varepsilon^4}b^2{x_1}^2\right)\psi + \varepsilon^{\sigma}\lambda|\psi|^{2\sigma}\psi,
	\end{split}\end{equation}
where $\sigma \in \mathbb{N}$ and $\lambda \in\R$, describing both focusing and defocusing nonlinearities. The sign of the nonlinearity will play no role for our analysis, since we shall prove below that 
the solution to \eqref{pde1} exists at least as long as the solution to a certain approximate equation does. Hence, we do not run into the problem of a possible finite time blow-up \cite{Mi}.

We want to analyze the behavior of the solution $\psi(t)$ in limit as $\eps\to 0_+$, assuming that the initial data for equation \eqref{pde1} is of the form
\[
\psi(0, x_1, x_2)=\varepsilon^{-1/2}\psi_0\left(\frac{x_1}{\varepsilon}, x_2\right), \ \text{where $\Vert \psi_0\Vert_{L^2(\mathbb{R}^2)} = 1$.}
\]
In other words, we assume that the initial wave function is already confined at the scale $\eps$ in the $x_1$-directions (an assumption which is consistent with the asymptotic regime considered). 
The fact that we have an effective coupling constant of the form $\lambda^\eps = \eps^\sigma \lambda$, with $\lambda \in \R$ fixed, then allows us to re-scale \eqref{pde1} 
via 
\begin{equation*}
x = \frac{x_1}{\eps}, \quad  y =x_2.
\end{equation*}
Keeping the $L^2$-norm of the rescaled solution $\psi^\eps( x, y) = \sqrt{\eps} \psi(x_1, x_2)$ fixed we thereby retain 
nonlinear effects of order $\mathcal O(1)$ in the rescaled nonlinear Schr\"odinger equation (NLS) for $\psi^\eps$. More precisely, $\psi^\eps$ solves 
\begin{equation}\begin{split}\label{pde2}
		i\partial_t \psi^\varepsilon = \frac{1}{\varepsilon^2}H\psi^\varepsilon - \frac{1}{2}\partial_{y}^2\psi^\varepsilon - \frac{ib}{\varepsilon}x\partial_{y}\psi^\varepsilon + \lambda|\psi^\varepsilon|^{2\sigma}\psi^\varepsilon,
		\quad \psi^\eps_{\mid t=0} = \psi_0(x,y),
	\end{split}\end{equation}
with $\eps$-independent initial data $\psi_0$. Here, and in the following, we denote by 
\[
H = -\frac{1}{2}\partial_{x}^2 + \frac{1}{2}b^2x^2
\]
the 1D quantum harmonic oscillator. The operator $H$ is essentially self-adjoint on $C_0^\infty(\R_x)\subset L^2(\mathbb{R}_x)$ with purely discrete spectrum 
\begin{equation}\label{spec}
{\rm spec}(H) = \left\{E_n=b(n+\tfrac{1}{2})\, : \, n\in\mathbb{N}_0\right\}.
\end{equation}
Due to the pre-factor $\tfrac{1}{\eps^2}$ in \eqref{pde2}, we expect a strong confining effect for $\psi^\eps$ in the $x$-direction. 
However, a non-trivial problem is introduced via the singular term $\propto \,\tfrac{1}{\eps}$, which 
does {\it not} commute with $H$ and which will have a profound effect on the effective dynamics (as we shall see).

Ignoring this issue for the moment, we shall proceed by filtering out the strong oscillations induced by $H$. We consequently expect that the filtered 
unknown admits a limit
	\begin{equation*}
		\phi^\varepsilon(t,x,y) := e^{itH/\varepsilon^2}\psi^\varepsilon(t,x,y)\stackrel{\eps\rightarrow 0_+ }{\longrightarrow}\phi(t,x,y),
	\end{equation*}
where $\phi$ solves an effective, $\eps$-independent equation. 
To formally derive this equation, we first note that $\phi^\eps$ solves 
\begin{equation}\label{filteredPDE}
		i\partial_t \phi^\varepsilon = -\frac{1}{2}\partial_{y}^2\phi^\varepsilon - \frac{ib}{\varepsilon}G\left(\frac{t}{\varepsilon^2}, \partial_{y}\phi^\varepsilon\right) + \lambda F\left(\frac{t}{\varepsilon^2}, \phi^\varepsilon\right), 
		\quad \phi^\eps_{\mid t=0} = \psi_0(x,y),
	\end{equation}
where 	
\begin{equation*}%\label{GF}
%\begin{aligned}
		G(\theta, u) :=\, e^{i\theta H}\left(xe^{-i\theta H}u\right),\quad
		F(\theta, u) :=\, e^{i\theta H}\left(\left|e^{-i\theta H}u\right|^{2\sigma}e^{-i\theta H}u\right).
%	\end{aligned}
	\end{equation*}
Next, we introduce the following Sobolev-type spaces for $m \in \mathbb N$:
	\begin{equation}\begin{split}\label{norm}
		\Sigma^m =&\, \Big\{u\in L^2(\mathbb{R}^2) \, : \, \Vert u\Vert_{\Sigma^m} < \infty\Big\},\\
		\Vert u\Vert_{\Sigma^m}^2: =&\, \Vert u\Vert_{L^2}^2 + \Vert H^{m/2}u\Vert_{L^2}^2 + \Vert \partial_{y}^mu\Vert_{L^2}^2.
	\end{split}\end{equation}
The choice $\Sigma\equiv \Sigma ^1$ thereby corresponds to the physical energy space associated with (\ref{pde2}).
For $m\ge2$, $\Sigma^m$ is a Banach algebra of the same type as commonly used for magnetic NLS systems, see \cite{dele, sparber, Mi}. (We note however, that our definition differs 
from the one in \cite{sparber} since it does not include any moments in the $y$-direction.)
It is then readily seen that $F\in C(\R\times\Sigma^2,\Sigma^2)$. Moreover, we can rewrite it in the form
\begin{equation*}\label{eq:F}
\begin{aligned}
F(\theta,u)=e^{i\theta (H-b/2)}\left(\left|e^{-i\theta (H-b/2)}u\right|^{2\sigma}e^{-i\theta (H-b/2)}u\right).
\end{aligned}
\end{equation*}
In view of \eqref{spec}, the operator $e^{i\theta (H-b/2)}$ is $\frac{2\pi}{b}$-periodic with respect to $\theta$, and hence so is $F$. We shall denote its average by 
\begin{equation}\label{eq:average}
\begin{split}
F_{\rm av}(u):= &\, \lim_{T \to \infty} \frac{1}{T} \int_0^T F\left(\theta, u\right) d\theta\\
= &\, \frac{b}{2\pi}\int_0^{\frac{2\pi}{b}}e^{i\theta H}\left(\left|e^{-i\theta H}u\right|^{2\sigma}e^{-i\theta H}u\right)d\theta.
\end{split}
\end{equation} 
We consequently expect that in the limit $\eps \to 0_+$: $F\left(\tfrac{t}{\eps^2}, \phi^\eps \right)$ converges (in some sense to be made precise) to $F_{\rm av}(\phi)$.

Coming back to the issue presented by the singular term $\propto \mathcal O(\tfrac{1}{\eps})$, a first key insight is that
\[
G_{\rm av}(u) := \frac{b}{2\pi}\int_0^{\frac{2\pi}{b}}G(\theta, u)d\theta \equiv 0.
\]
To see this, let $\chi_n(x)$ be an eigenfunction of $H$ associated to 
the $n$-th eigenvalue $E_n$, and denote the inner product on $L^2(\mathbb{R}_x)$ by ${\langle\cdot,\cdot\rangle}_{L^2_x}$. 
Then, we can compute
	\begin{equation}\begin{aligned}\label{Gav}
		G_{\rm av}(u) =&\, \frac{b}{2\pi}\int_0^{\frac{2\pi}{b}}\sum_{n = 0}^\infty e^{i\theta H}\left(xe^{-i\theta b(n+\frac{1}{2})}{\langle u, \chi_n\rangle}_{L^2_x}\chi_n\right)d\theta\\
		=&\, \frac{b}{2\pi}\int_0^{\frac{2\pi}{b}}\sum_{n = 0}^\infty\sum_{m = 0}^\infty e^{i\theta b(m+\frac{1}{2})}{\langle x\chi_n, \chi_m\rangle}_{L^2_x}
		e^{-i\theta b(n+\frac{1}{2})}{\langle u, \chi_n\rangle}_{L^2_x}\chi_m \, d\theta\\
		=&\, \frac{b}{2\pi}\int_0^{\frac{2\pi}{b}}\sum_{n = 0}^\infty\sum_{m \neq n}^\infty e^{i\theta b(m-n)}\langle x\chi_n, \chi_m\rangle_{L^2_x}{\langle u, \chi_n\rangle}_{L^2_x}\chi_m\, d\theta.
	\end{aligned}
\end{equation}
In the last line we have used the fact that $\langle x\chi_n, \chi_n\rangle_{L^2_x} = 0$, since $x\chi_n^2(x)$ is an odd function. 
Interchanging the order of summation and integration formally yields $G_{\rm av}(u) = 0$, by periodicity. One therefore expects the singular term to be at worst 
of order $\mathcal O(1)$, as $\eps \to 0_+$. 
Indeed, we shall prove in Section \ref{sec:filter} that for sufficiently regular $\phi^\eps$:
\[
		- \frac{ib}{\varepsilon}\int_0^tG\left(\frac{s}{\varepsilon^2}, \partial_{y}\phi^\varepsilon(s)\right)ds = \frac{1}{2}\int_0^t\partial_{y}^2\phi^\varepsilon(s)ds + \mathcal{O}(\eps),\quad \text{as $\eps \to 0_+$,}
\]
which combined with \eqref{filteredPDE} (in integral form) yields
\begin{equation*}\label{pdeSimp}
	\phi^\varepsilon = \psi_0 - i\lambda \int_0^tF\left(\frac{s}{\varepsilon^2}, \phi^\varepsilon(s)\right)ds + \mathcal{O}(\eps).
\end{equation*}
In the limit $\eps \to 0_+$, we thus (at least formally) obtain the following effective model:
\begin{equation}\label{limeq}
		i\partial_t \phi = \lambda F_{\rm av}(\phi),\quad \phi_{\mid t=0}= \psi_0(x,y).
\end{equation}
As a result of the strong magnetic confinement, this equation is seen to be no longer dispersive. It also involves a non-local, nonlinear self-interaction. 
The first main result of this paper can now be stated as follows:

\begin{theorem}\label{main}
Let $\sigma \in \mathbb{N}$ and $\lambda \in \R$. 

{\rm (i)} For all $\psi_0 \in \Sigma^{15}$ there exists $T_{\rm max} \in (0,\infty]$ such that the effective model \eqref{limeq} has a unique 
maximal solution $\phi\in C^1([0,T_{\rm max}),\Sigma^{15})$, depending continuously on the initial data.

{\rm (ii)} For all $T \in (0, T_{\rm max})$, there exist $\varepsilon_T > 0$ and $C_T > 0$ such that for all $\varepsilon \in (0,\varepsilon_T)$, the original equation \eqref{pde2} 
has a unique solution $$\psi^\varepsilon \in C([0,T], \Sigma^{15})\cap C^1([0,T],\Sigma^{13}),$$ which satisfies
	\begin{equation*}
		\sup_{t \in [0,T]}\big \Vert \psi^{\varepsilon} (t,  \cdot)- e^{-itH/\varepsilon^2}\phi(t,  \cdot) \big \Vert_{L^2} \leq C_T\varepsilon.
	\end{equation*}
Furthermore, $\Vert\psi^\varepsilon\Vert_{L^\infty([0,T], \Sigma^{15})}$ is uniformly bounded with respect to $\varepsilon$.
\end{theorem}
Here, the rather strong regularity assumption $\psi_0 \in \Sigma^{15}$ is a consequence of our analysis, which requires sufficient smoothness to retain the sharp 
$\mathcal{O}(\eps)$ asymptotic convergence rate between the true and the approximate solution. We will show how to lessen this requirement in our second theorem below. 
It is crucial for our analysis, however, that the vector potential $\beta(x_1)$ is an odd function. Without this fact, $G_{\rm av}(u) \not = 0$ (in general)
and the term $\propto \,\mathcal O(\tfrac{1}{\eps})$ does not converge in any sense as $\varepsilon \to 0$.

Equation \eqref{limeq} shows that strong magnetic fields suppress all dispersive effects in $\mathbb R^2$. However, there is a 
qualitative difference in the type of confinement with respect to the $x$- versus the $y$-direction. 
This becomes more apparent if, instead of a fixed coupling constant $\lambda\in \R$, 
we allow for a more general, inhomogeneous coupling given by some function $\lambda = \lambda(x_1, x_2)$ such that 
\[
\|\lambda \|_{L^\infty}= \mathcal O(1), \quad \text{as $\eps \to 0_+$.}
\] 
Using the same rescaling as before, we obtain, instead of \eqref{pde2}, the slightly more general NLS
\begin{equation}\label{pde2alt}
		i\partial_t \psi^\varepsilon = \frac{1}{\varepsilon^2}H\psi^\varepsilon - \frac{1}{2}\partial_{y}^2\psi^\varepsilon 
		- \frac{ib}{\varepsilon}x\partial_{y}\psi^\varepsilon + \lambda(\eps x, y)|\psi^\varepsilon|^{2\sigma}\psi^\varepsilon,\quad \psi^\eps_{\mid t=0}= \psi_0(x,y).
\end{equation}
Note that for $\eps>0$ the coupling function $\lambda(\eps x, y)$ 
does not commute with the action of the confinement Hamiltonian $H$. However, the fact that $\lambda$ is slowly varying in $x$ will allow us to 
obtain a similar effective model equation as before. 

\begin{theorem}\label{mainCor}
Let $\sigma \in \mathbb N$ and $\lambda \in W^{15, \infty}(\R^2)$. 

{\rm (i)} For all $\psi_0 \in \Sigma^2$, there exists $T_{\rm max} \in (0,\infty]$ and a unique maximal solution $\phi \in C^1([0, T_{\rm max}), \Sigma^2)$ to the effective model
	\begin{equation}\label{limeqalt2}
		i\partial_t \phi = \lambda(0,y) F_{\rm av}(\phi),\quad \phi_{\mid t=0} = \psi_0(x,y).
	\end{equation}

{\rm (ii)} For all $T \in (0, T_{\rm max})$, equation \eqref{pde2alt} 
has a unique solution $$\psi^\varepsilon \in C([0,T], \Sigma^2)\cap C^1([0,T],L^2(\R^2)),$$ 
which satisfies
	\begin{equation}\label{mainCorCon}
		\lim_{\eps\to 0}\left(\sup_{t \in [0,T]}\big \Vert \psi^{\varepsilon} (t,  \cdot)- e^{-itH/\varepsilon^2}\phi(t,  \cdot) \big \Vert_{L^2}\right) = 0.
	\end{equation}
Furthermore, $\Vert\psi^\varepsilon\Vert_{L^\infty([0,T], \Sigma^2)}$ is uniformly bounded with respect to $\varepsilon$.
\end{theorem}

Equation \eqref{limeqalt2} shows that nonlinear effects are only relevant along the $y$-axis. 
In particular, if the function $\lambda$ vanishes along this axis, the effective model for $\phi$ becomes trivial. We also see that by lowering the regularity of our initial data $\psi_0$, we lose the 
$\mathcal O(\eps)$ convergence rate and are only able to prove a limiting statement. The optimal regularity requirements for which one can retain an asymptotic rate remains an open problem at this point.
We also note that the condition $\lambda \in W^{15, \infty}(\R^2)$ is most likely far from optimal, but this requirement appears as an artifact of our analysis. Furthermore, this choice of $\lambda$ has no effect on the convergence rate of \eqref{mainCorCon}; one can achieve $\mathcal O(\eps)$-convergence if we assume $\psi_0 \in \Sigma^{15}$. We chose to provide this result in Theorem \ref{mainCor} rather than Theorem \ref{main} for simplicity of presentation.

Finally, we shall turn to the question of initial data which are initially concentrated in an eigenspace of the confinement Hamiltonian $H$. It turns out that in this case, the 
solution $\phi$ remains concentrated in this eigenspace. As a result, the effective 
model becomes a rather simple ordinary differential equation which can be solved explicitly.

\begin{corollary}\label{eigenThm}
Let $\phi$ be the solution to \eqref{limeqalt2} and suppose the initial data $\psi_0$ is given by $\psi_0(x,y) = \alpha_0(y) \chi_n(x)$, where $\alpha_0(y)\in \C$ is some given amplitude 
and $\chi_n$ is an eigenfunction to the $n$-th eigenvalue $E_n$ of $H$. Then, for all $t \in [0, T_{\rm max})$,
\begin{equation*}%\label{PnPhi}
\phi(t,x,y)= \alpha_0(y)\chi_n(x)e^{-it\omega_n(y)} ,
\end{equation*}
i.e. a time-periodic state with frequency $$\omega_n(y) := \lambda(0,y)|\alpha_0(y)|^{2\sigma}\|\chi_n\|_{L_x^{2\sigma+2}}^{2\sigma+2}.$$
\end{corollary}

The results above are derived using high-frequency averaging techniques similar to those in \cite{mehat, MeSp, dele, sparber}, the last two being 
the works most closely related to ours. 
Note however, that \cite{dele} deals with a 3D situation with combined electric and magnetic confinement, and 
considers self-consistent interactions via the Poisson equation (instead of power-law nonlinearities). In this case, the effective model 
is found to be an infinite system of nonlinearly coupled PDEs. An analogous system is obtained in our case if we 
decomposed the solution $\phi$ to \eqref{limeq} w.r.t. $x\in \R$ using the orthogonal basis of eigenfunctions $(\chi_n)_{n\in \mathbb N}$ of $H$.
  
When comparing  the present work with \cite{sparber}, we first note that we use a different gauge and henceforth a different (non-isotropic) scaling w.r.t. to $(x_1, x_2)\in \R^2$. 
One can transfer our Hamiltonian $\mathcal H$, given by \eqref{H} and \eqref{B}, into the one appearing in \cite{sparber} by conjugating 
\[
\mathcal H\mapsto e^{-i S_\eps} \, \mathcal H \mathcal \, e^{iS_\eps}, \quad \text{where $S_\eps (x_1, x_2)= \tfrac{b x_1 x_2}{2\eps^2} $}.
\]
In turn, this gauge transform maps $\psi^\eps \mapsto e^{-i S_\eps}\psi^\eps$. Doing so, however, results in 
initial data $\psi_0\mapsto e^{-i S_\eps}\psi_0$, which are highly oscillatory and no longer bounded in $\Sigma$ uniformly w.r.t. $\eps$. 
Such initial data are excluded by assumption in \cite{sparber} and hence our results cannot be inferred from those in \cite{sparber}, and vice versa. 
One should also note that the singular phase $S_\eps$ does not commute with the action of $H$. It is therefore not surprising that the resulting 
limiting model \eqref{limeq} is different from the one obtained in \cite{sparber} in that it projects 
onto a different spectral subspace of $L^2(\R^2)$. From a technical point of view, \cite{sparber} also does not need to deal with the appearance of the additional 
singular term $\propto \,\frac{1}{\eps}$. It may be of further interest to study the relationship between the choice of gauge 
and the resulting effective dynamics, but such an analysis is beyond the scope of this paper.

Finally, we recognize that the linear part of our model is given by a quadratic differential operator (in position and momentum), and thus 
some parts of our analysis may be replaced using classical-quantum correspondence theorems (Egorov-type results). 
However, our averaging techniques are in principle also applicable to non-quadratic differential operators and thus may apply to a wider class of 
magnetic Schr\"odinger equations than studied here.

\smallskip

The rest of this paper is now organized as follows: In Section \ref{sec:frame}, we collect several a-priori estimates and other technical results related to the functional framework. 
After that, we will derive appropriate (local) well-posedness results for both the original NLS and the limiting equations in Section \ref{sec:LWP}, proving item (i) of 
both Theorem \ref{main} and Theorem \ref{mainCor}. 
In Section \ref{sec:filter} we shall then rigorously derive the
asymptotic expansion for the solution of the filtered equation \eqref{filteredPDE}.
The nonlinear stability of the asymptotic approximation 
is then proved in Section \ref{sec:stab}, completing the proof of Theorem \ref{main}. In Section \ref{sec:stabCor}, we will 
demonstrate the additional steps needed to prove Theorem \ref{mainCor}. Finally, Corollary \ref{eigenThm} will be proved in Section \ref{sec:eigen} and we shall also 
present details of a somewhat technical result (concerning the equivalence of certain norms) in the appendix.

%%%%%%%%%%%%%%%%%%%%%%%%%%%%%%%%%%%%%%%%%%%

\section{\textbf{Functional Framework}}\label{sec:frame}

This section presents several technical results and estimates to be used later on. We begin by recalling the following lemma, which is proved under more generality in \cite{mehat}. 
Note that here and in the rest of the paper, we write
\[
f(u) \lesssim g(u) \text{  to mean } f(u) \leq Cg(u),
\]
where $C$ is some positive constant which does not depend on $u$ or $\eps$.

\begin{lemma}\label{1}
For all $u \in \Sigma^{m+1}$,
	\begin{equation*}
		\Vert H^{1/2}u\Vert^2_{\Sigma^m} + \Vert xu\Vert^2_{\Sigma^m}  + \Vert \partial_xu\Vert^2_{\Sigma^m} + \Vert \partial_yu\Vert^2_{\Sigma^m} \lesssim \Vert u\Vert^2_{\Sigma^{m+1}}.
	\end{equation*}
\end{lemma}
The proof of this follows along the same lines as the one for \cite[Lemma 2.3]{mehat}. Note however, that in \cite{mehat}, the authors consider energy spaces and operators $H$ which include additional 
weights in the $y$-direction. In the present setting we can simply ignore these additional terms. 

For our purposes, it will be also be convenient to consider an $\eps$-dependent version of the 
energy-norm. Namely, we define
	\begin{equation*}\begin{split}%\label{equivNorm}
		\Vert u\Vert_{\Sigma_\varepsilon^m}^2 :=&\, \Vert u\Vert_{L^2}^2 + \Vert H_\varepsilon^{m/2}u\Vert_{L^2}^2 + \Vert \partial_{y}^mu\Vert_{L^2}^2, 
	\end{split}\end{equation*}
where
	\begin{equation}\label{Heps}
		H_\varepsilon := H - i\varepsilon bx\partial_y - \frac{\varepsilon^2}{2}\partial^2_y, \quad \text{for $\eps>0$}.
	\end{equation}
In the next lemma (analogous to \cite[Lemma 2.4]{dele}), we will show that $\Vert\cdot\Vert_{\Sigma_\varepsilon^m}$ is  indeed equivalent to $\Vert\cdot\Vert_{\Sigma^m}$, given  in \eqref{norm}:
\begin{proposition}\label{NormEquiv}
For all $m \in \N$ there exist $\varepsilon_m \in (0, 1]$ such that for all $\varepsilon \in (0, \varepsilon_m]$,
	\begin{equation*}
		\frac{1}{2}\Vert u\Vert_{\Sigma^m}^2 \leq \Vert u\Vert_{L^2}^2 + \Vert H_\varepsilon^{m/2}u\Vert_{L^2}^2 + \Vert \partial_y^{m}u\Vert_{L^2}^2 \leq 2\Vert u\Vert_{\Sigma^m}^2.
	\end{equation*}
\end{proposition}
\begin{proof}
The proof of this result is somewhat lengthy. For the sake of presentation it will be given in Appendix A.
\end{proof}
With this equivalence in hand, we can now prove the following estimates, which will be used several times later on:
\begin{lemma}\label{Ubound}
Let $u \in \Sigma^m$. For all $t \in \mathbb{R}$ and $\varepsilon \in (0, \varepsilon_m]$, it holds
\[
\big \Vert e^{-itH}u \big \Vert_{\Sigma^m} = \Vert u \Vert_{\Sigma^m},\quad \big\Vert e^{i\frac{t}{2} \partial_{y}^2}u \big \Vert_{\Sigma^m} = \Vert u \Vert_{\Sigma^m},
\]
as well as
\[\big \Vert e^{-itH_\varepsilon}u \big \Vert_{\Sigma^m} \lesssim \Vert u \Vert_{\Sigma^m}.\]
\end{lemma}
\begin{proof}
Here the first line follows trivially from the fact that both $H$ and $\partial_{y}^2$ commute with each term in the $\Sigma^m$-norm, defined in \eqref{norm}. 
The second line follows from the fact that $H_\varepsilon$ commutes with each term in the $\Sigma^m_\varepsilon$-norm and Proposition \ref{NormEquiv} above.
\end{proof}

\begin{remark}
If we would have allowed for additional weights in the $y$-direction in definition \eqref{norm}, the operators $\partial_{y}^2$ and $H_\eps$ would no longer commute with each term in the $\Sigma^m$- and $\Sigma^m_\varepsilon$-norms, respectively, 
spoiling the uniform (in $\eps$) bounds obtained above. 
\end{remark}

As a final  preparatory step, we derive the following Moser-type estimate for our nonlinearity (with possibly $(x,y)$-dependent coupling function $\lambda$):
\begin{lemma}\label{Moser}
Let $m \in \N$. For all $\lambda \in W^{m,\infty}(\mathbb{R}^2)$ and $u \in \Sigma^m$,
	\begin{equation*}
		\Vert\lambda|u|^{2\sigma}u\Vert_{\Sigma^m} \lesssim \Vert \lambda\Vert_{W^{m,\infty}}\Vert u\Vert_{L^\infty}^{2\sigma}\Vert u\Vert_{\Sigma^m}.
	\end{equation*}
\end{lemma}
\begin{proof}
This estimate can be proven by bounding each term in the $\Sigma^m$-norm separately. Clearly,
	\begin{equation*}
		\Vert\lambda|u|^{2\sigma}u\Vert_{L^2} \lesssim \Vert \lambda\Vert_{W^{m,\infty}}\Vert u\Vert_{L^\infty}^{2\sigma}\Vert u\Vert_{\Sigma^m}.
	\end{equation*}
For the terms involving $x^{m}$, one similarly has
	\begin{equation*}
		\Vert x^{m}\lambda|u|^{2\sigma}u\Vert_{L^2} \leq \Vert \lambda\Vert_{W^{m,\infty}}\Vert u\Vert_{L^\infty}^{2\sigma}\Vert x^{m}u\Vert_{L^2} \lesssim \Vert \lambda\Vert_{W^{m,\infty}}\Vert u\Vert_{L^\infty}^{2\sigma}\Vert u\Vert_{\Sigma^m},
\end{equation*}
which follows from Lemma \ref{1}.
The remaining derivative terms can be bounded with the following computation. First, by the Leibniz rule we have
	\begin{equation}\begin{aligned}\label{MosEq3}
		\Vert \partial_y^{m}(\lambda|u|^{2\sigma}u)\Vert_{L^2} &= \Big\Vert \sum_{k=0}^m {m \choose k}\partial_y^{m-k}\lambda\partial_y^{k}|u|^{2\sigma+1}\Big\Vert_{L^2}\\
		&\leq \sum_{k=0}^m {m \choose k}\Vert \lambda\Vert_{W^{m-k,\infty}}\Vert\partial_y^{k}|u|^{2\sigma+1}\Vert_{L^2}.
	\end{aligned}\end{equation}
Applying the Leibniz rule again to the part involving $u$ gives
	\begin{equation}\begin{split}\label{MosEq1}
		\Vert\partial_y^{k}|u|^{2\sigma+1}\Vert_{L^2} &= \Big\Vert\sum_{k_1+...+k_{2\sigma+1}=k}{k \choose {k_1,...,k_{2\sigma+1}}}\prod_{j=1}^{2\sigma+1}\partial_y^{k_j}|u|\Big\Vert_{L^2}\\
		&\leq \sum_{k_1+...+k_{2\sigma+1}=k}{k \choose {k_1,...,k_{2\sigma+1}}}\Big\Vert\prod_{j=1}^{2\sigma+1}\partial_y^{k_j}|u|\Big\Vert_{L^2},
	\end{split}\end{equation}
where the sum is over all combinations of non-negative integers $k_j$ such that $k_1+...+k_{2\sigma+1}=k$. We then apply H\"older's inequality:
	\begin{equation}
		\Big\Vert\prod_{j=1}^{2\sigma+1}\partial_y^{k_j}|u|\Big\Vert_{L^2} \leq \prod_{j=1}^{2\sigma+1}\Vert \partial_y^{k_j}|u|\Vert_{L^{2k/k_j}}.
	\end{equation}
Finally, we can apply a Gagliardo-Nirenberg inequality to obtain
	\begin{equation}\begin{split}\label{MosEq2}
		\prod_{j=1}^{2\sigma+1}\Vert \partial_y^{k_j}|u|\Vert_{L^{2k/k_j}} &\lesssim \prod_{j=1}^{2\sigma+1}\Vert \partial_y^{k}|u|\Vert_{L^2}^{k_j/k}\Vert u\Vert_{L^\infty}^{1-k_j/k}\\
		&\lesssim \prod_{j=1}^{2\sigma+1}\Vert u\Vert_{\Sigma^m}^{k_j/k}\Vert u\Vert_{L^\infty}^{1-k_j/k} = \Vert u\Vert_{\Sigma^m}\Vert u\Vert_{L^\infty}^{2\sigma}.
	\end{split}\end{equation}
In summary, (\ref{MosEq1})--(\ref{MosEq2}) yield $\Vert\partial_y^{k}|u|^{2\sigma+1}\Vert_{L^2} \lesssim \Vert u\Vert_{\Sigma^m}\Vert u\Vert_{L^\infty}^{2\sigma}$. Plugging this into (\ref{MosEq3}) and noting that $\Vert \lambda\Vert_{W^{m-k,\infty}} \lesssim \Vert \lambda\Vert_{W^{m,\infty}}$ gives
	\begin{equation*}
		\Vert \partial_y^{m}(\lambda|u|^{2\sigma}u)\Vert_{L^2} \lesssim \Vert \lambda\Vert_{W^{m,\infty}}\Vert u\Vert_{\Sigma^m}\Vert u\Vert_{L^\infty}^{2\sigma},
	\end{equation*}
completing the proof.
\end{proof}

\begin{corollary}\label{MoserCor}
For all $m \geq 2$,
	\begin{equation*}
		\Vert \lambda|u|^{2\sigma}u\Vert_{\Sigma^m} \lesssim \Vert \lambda \Vert_{W^{m,\infty}}\Vert u\Vert_{\Sigma^m}^{2\sigma+1}.
	\end{equation*}
\end{corollary}
\begin{proof}
This follows immediately from Lemma \ref{Moser} and the fact that $\Sigma^m$ embeds continuously into $L^\infty$ for $m \geq 2$.
\end{proof}

%%%%%%%%%%%%%%%%%%%%%%%%%%%%%%%%%%%%%%%%%%%

\section{\textbf{Well-Posedness Results}}\label{sec:LWP}

In this section, we shall provide local well-posedness results for (\ref{pde2alt}) and (\ref{limeqalt2}) under sufficient generality for both Theorem \ref{main} and Theorem \ref{mainCor}.
For the remainder of this work, we shall always assume $\varepsilon \in (0, \varepsilon_{15}]$, where $\varepsilon_{15}$ is defined as in Proposition \ref{NormEquiv}.

\begin{proposition}\label{wp1}
Fix $\varepsilon > 0$, and assume $\psi_0 \in \Sigma^m$, $\lambda \in W^{m,\infty}$ with $m \geq 2$. Then, there exists $T^\varepsilon_{\rm max} \in (0,\infty]$ and a unique maximal solution 
\[\psi^\varepsilon \in C([0,T^\varepsilon_{\rm max}), \Sigma^{m})\cap C^1([0,T^\varepsilon_{\rm max}), \Sigma^{m-2})\] to \eqref{pde2alt}, 
depending continuously on the initial data $\psi_0$.
\end{proposition}
\begin{proof}
The Duhamel formulation of (\ref{pde2alt}) is
	\begin{equation}\label{duhamel}\begin{aligned}
		\psi^\varepsilon(t)&\, = e^{-itH_\varepsilon/\varepsilon^2}\psi_0 - i\int_0^te^{-i(t-s)H_\varepsilon/\varepsilon^2}\lambda(\varepsilon x, y)|\psi^\varepsilon(s)|^{2\sigma}\psi^\varepsilon(s)ds\\
		& \, := \Psi(\psi^\varepsilon)(t),
	\end{aligned}\end{equation}
where, in view of \eqref{Heps}, we have
\[
e^{-itH_\varepsilon/\varepsilon^2}= e^{-it(H/\varepsilon^2 - \frac{1}{2}\partial^2_{y} - ibx\partial_{y}/\varepsilon)}.
\] 
We shall prove that for $T^\varepsilon > 0$ small enough, $\Psi$ is a contraction on a suitable ball $B_R(0) \subset C([0,T^\varepsilon], \Sigma^m)$. A fixed point argument then proves existence of 
a unique solution in $C([0,T^\varepsilon], \Sigma^m)$, and continuous dependence on initial data follows easily via Gr\"onwall's inequality. 
To obtain the contraction property we use the fact that $e^{-itH_\varepsilon/\varepsilon^2}$ is bounded on $\Sigma^m$ (see Lemma \ref{Ubound}) to infer
\begin{equation*}\begin{aligned}
		\Vert\Psi(\psi^\varepsilon)(t)\Vert_{\Sigma^m} &\leq \Vert e^{-itH_\varepsilon/\varepsilon^2} \psi_0\Vert_{\Sigma^m} + \int_0^t\Vert e^{-i(t-s)H_\varepsilon/\varepsilon^2} \lambda(\varepsilon x, y)|\psi^\varepsilon(s)|^{2\sigma}\psi^\varepsilon(s)\Vert_{\Sigma^m}ds\\
		&\lesssim \Vert \psi_0\Vert_{\Sigma^m} + \int_0^t\Vert \lambda(\varepsilon x, y)|\psi^\varepsilon(s)|^{2\sigma}\psi^\varepsilon(s)\Vert_{\Sigma^m}ds\\
		&\lesssim \Vert \psi_0\Vert_{\Sigma^m} + \int_0^t\Vert \lambda\Vert_{W^{m,\infty}}\Vert \psi^\varepsilon(s)\Vert_{\Sigma^m}^{2\sigma+1}ds,
	\end{aligned}\end{equation*}
where the last line follows from Corollary \ref{MoserCor}. In view of this estimate, if we have existence at $T^\varepsilon$ then we can iterate this method to 
obtain existence up to some $\widetilde{T^\varepsilon} > T^\varepsilon$. Thus, there must exist a maximal $T^\varepsilon_{\rm max} \in (0,\infty]$ for which we have local well-posedness in 
$C([0,T^\varepsilon_{max}), \Sigma^m)$. Furthermore if $T^\varepsilon_{\rm max} < \infty$ we must have
	\begin{equation}\label{blowup}
		\lim_{t\to T^\varepsilon_{\rm max}} \Vert\psi^\varepsilon(t)\Vert_{\Sigma^m} = \infty.
	\end{equation}
Finally, given a solution $\psi^\varepsilon \in C([0,T^\varepsilon_{\rm max}), \Sigma^m)$ to (\ref{pde2alt}), Lemma \ref{1} also yields
	\begin{equation}\begin{aligned}\label{wpEst1}
		\sup_{t\in[0,T^\varepsilon]}\Vert \partial_t\psi^\varepsilon\Vert_{\Sigma^{m-2}} &\leq \sup_{t\in[0,T^\varepsilon]} \Big(\frac{1}{\varepsilon^2}\Vert H\psi^\varepsilon\Vert_{\Sigma^{m-2}} + \frac{1}{2}\Vert\partial_{y}^2\psi^\varepsilon\Vert_{\Sigma^{m-2}}\\
		&\quad\quad\quad\quad+ \frac{b}{\varepsilon}\Vert x\partial_{y}\psi^\varepsilon\Vert_{\Sigma^{m-2}} + \Vert\lambda(\varepsilon x, y)|\psi^\varepsilon|^{2\sigma}\psi^\varepsilon\Vert_{\Sigma^{m-2}}\Big)\\
		&\lesssim (1+\frac{1}{\varepsilon^2})\sup_{t\in[0,T^\varepsilon]} \Big(\Vert \psi^\varepsilon\Vert_{\Sigma^m}+\Vert \psi^\varepsilon\Vert_{\Sigma^m}^{2\sigma+1}\Big) < \infty
	\end{aligned}\end{equation}
for all $T^\varepsilon < T^\varepsilon_{\rm max}$ and for each fixed $\varepsilon$. Thus, $\psi^\varepsilon  \in C^1([0,T^\varepsilon_{\rm max}), \Sigma^{m-2})$.
\end{proof}

\begin{remark}
Since $\psi_0$ is independent of $\varepsilon$, we could in principle bound $\| \psi^\eps(t, \cdot)\|_{\Sigma^m}$ up to some $\varepsilon$-independent $T>0$. 
This would give well-posedness in $C([0,T], \Sigma^m)$ uniformly in $\varepsilon$, and henceforth eliminate the possibility that $T^\varepsilon_{\rm max} \to 0$ as $\varepsilon \to 0_+$. 
However, we will later need well-posedness in $C([0,T^\varepsilon_{\rm max}), \Sigma^{m})\cap C^1([0,T^\varepsilon_{\rm max}), \Sigma^{m-2})$, which cannot be obtained uniformly in $\eps$ with the above techniques.
Eventually, all of this will be overcome by the fact that we shall prove $T^\varepsilon_{\rm max} \geq T_{\rm max}$, i.e. the $\varepsilon$-independent maximal existence time of the 
limiting model \eqref{limeqalt2}.
\end{remark}

Next, we shall turn to the limiting model \eqref{limeqalt2} and establish its local well-posedness, and thus prove item (i) of Theorem \ref{main} and Theorem \ref{mainCor}.

\begin{proposition}\label{wp2}
Let $\psi_0 \in \Sigma^m$ and $\lambda \in W^{m,\infty}$ with $m \geq 2$. Then, there exists $T_{\rm max} > 0$ and a unique maximal solution $\phi \in C^1([0,T_{\rm max}), \Sigma^{m})$ 
to \eqref{limeqalt2}, 
depending continuously on the initial data $\psi_0$.
\end{proposition}
\begin{proof}
Integrating \eqref{limeqalt2} w.r.t. $t$ yields
	\begin{equation}\label{limeqDu}
		\phi(t) = \psi_0 - i\int_0^t\lambda(0,y)F_{\rm av}(\phi(s))ds := \Phi(\phi)(t).
	\end{equation}
As in the proof of the previous result, we can use a fixed point argument on the solution map to obtain existence in $C([0,T_{\rm max}), \Sigma^m)$. This follows from the estimate
	\begin{equation*}\begin{aligned}
		\Vert\Phi(\phi)(t)\Vert_{\Sigma^m} &\leq \Vert \psi_0\Vert_{\Sigma^m} +  \int_0^t\Vert \lambda(0,y)F_{\rm av}(\phi(s))\Vert_{\Sigma^m}ds\\
		&\leq \Vert \psi_0\Vert_{\Sigma^m} +  \int_0^t\frac{b}{2\pi}\int_0^{\frac{2\pi}{b}}\left\Vert \lambda(0,y)e^{i\theta H}\left(|e^{-i\theta H}\phi(s)|^{2\sigma}e^{-i\theta H}\phi(s)\right)\right\Vert_{\Sigma^m}d\theta ds\\
		&\lesssim \Vert \psi_0\Vert_{\Sigma^m} +  \int_0^t\Vert \lambda\Vert_{W^{m,\infty}}\Vert \phi(s)\Vert_{\Sigma^m}^{2\sigma+1}ds,
	\end{aligned}\end{equation*}
where we have again used Corollary \ref{MoserCor}.The fact that $\partial_t\phi \in \Sigma^{m}$ then follows from the estimate
\[
	\Vert \partial_t\phi(t)\Vert_{\Sigma^m} \lesssim \Vert \lambda\Vert_{W^{m,\infty}}\Vert \phi(t)\Vert_{\Sigma^m}^{2\sigma+1}.
\]
\end{proof}

Propositions \ref{wp1} and \ref{wp2} cover well-posedness for both $\psi_0 \in \Sigma^{15}$ and $\psi_0 \in \Sigma^2$. However, to prove item (ii) in Theorem \ref{mainCor} we will need to approximate $\Sigma^2$-solutions with $\Sigma^{15}$-solutions. In light of the continuous dependence on initial data, this can be achieved by approximating $\psi_0 \in \Sigma^2$ with a sequence of regularized initial 
$(\psi_0^\eta)_{\eta\ge 0} \subset \Sigma^{15}$. We show how this can be done in the following lemma:
%%%%%
\begin{lemma}\label{regLem}
Suppose $\psi_0 \in \Sigma^2$. Let $\eta > 0$, $m \in \N$, and denote 
	\[
	\psi^{\eta}_0 = (1 + \eta H)^{-m/2}(1 - \eta \partial^2_y)^{-m/2}\psi_0.
	\] 
	Then, $\psi^{\eta}_0 \in \Sigma^{2+m}$ and we have the following estimates:
	\begin{equation}\label{lemreg1}
		\Vert \psi^{\eta}_0\Vert_{\Sigma^2} \leq \Vert \psi_0\Vert_{\Sigma^2},\quad \Vert \psi^{\eta}_0\Vert_{\Sigma^{2+m}} \lesssim (1 + \eta^{-m/2})\Vert \psi_0\Vert_{\Sigma^2},
	\end{equation}
as well  as
	\begin{equation}\label{lemreg3}
		\lim_{\eta \to 0}\Vert \psi^{\eta}_0 - \psi_0\Vert_{\Sigma^{2}} = 0.
	\end{equation}
\end{lemma}
\begin{proof}
To prove the first  inequality of \eqref{lemreg1}, we merely note that the operators $(1 + \eta H)^{-m/2}$ and
$(1 - \eta \partial^2_y)^{-m/2}$ are bounded by $1$ on $L^2(\R^2)$ and they commute with each term in the $\Sigma^{2}$-norm. To get the second inequality in \eqref{lemreg1}, we estimate
	\begin{equation*}\begin{split}
		\Vert H^{(2+m)/2}\psi^{\eta}_0\Vert_{L^2} &= \Vert H^{(2+m)/2}(1 + \eta H)^{-m/2}(1 - \eta \partial^2_y)^{-m/2}\psi_0\Vert_{L^2}\\
		&\leq \Vert H^{m/2}(1 + \eta H)^{-m/2}(H\psi_0)\Vert_{L^2}.
	\end{split}\end{equation*}
One can easily see that for any $z > 0$: $z^{m/2}(1 + \eta z)^{-m/2} \leq \eta^{-m/2}$.
Therefore, we have
	\begin{equation}\label{reg1}
		\Vert H^{(2+m)/2}\psi^{\eta}_0\Vert_{L^2} \leq \eta^{-m/2}\Vert H\psi_0\Vert_{L^2} \lesssim \eta^{-m/2}\Vert \psi_0\Vert_{\Sigma^2},
	\end{equation}
and similarly
	\begin{equation}\label{reg2}
		\Vert (-\partial^2_y)^{(2+m)/2}\psi^{\eta}_0\Vert_{L^2} \lesssim \eta^{-m/2}\Vert \psi_0\Vert_{\Sigma^2}.
	\end{equation}
Combining (\ref{reg1}) and (\ref{reg2}), and keeping in mind the definition of the $\Sigma^2$-norm, establishes the second inequality of \eqref{lemreg1}.

Finally, we shall prove (\ref{lemreg3}): let $u \in L^2(\mathbb{R}^2)$ and $\chi_n=\chi_n(x)$ be the $n$-{th} eigenfunction of the harmonic oscillator $H$. In the following, we 
denote by ${\langle \cdot, \cdot\rangle}_{L_x^2}$  the inner product in $L^2(\mathbb{R}_x)$. Moreover, 
$$\mathcal F: u(x,y) \mapsto \hat{u}(x,\xi)$$ denotes the 1D-Fourier transform with respect to $y\in \R$. Then, by Plancherel and the fact that $[\mathcal F, H]=0$:
	\begin{equation*}\begin{aligned}
		&\Vert u - (1 + \eta H)^{-m/2}(1 - \eta \partial^2_y)^{-m/2}u\Vert_{L^2}^2 \\
		& \ = \Vert (1 - (1 + \eta H)^{-m/2}(1 + \eta \xi^2)^{-m/2})\hat{u}\Vert_{L^2}^2\\
		&\ = \sum_{n=0}^\infty\int_{\mathbb{R}} \left|(1 - (1 + \eta E_n)^{-m/2}(1 + \eta \xi^2)^{-m/2}) {\langle \hat{u}, \chi_n \rangle}_{L_x^2} \right|^2d\xi.
	\end{aligned}\end{equation*}
Note that $|1 - (1 + \eta E_n)^{-m/2}(1 + \eta \xi^2)^{-m/2}| \leq 1$ for all $n \geq 0$, $\xi \in \mathbb{R}$, and $\eta > 0$. Moreover,
	\begin{equation*}
		\sum_{n=0}^\infty\int_{\mathbb{R}}| {\langle \hat{u}, \chi_n \rangle}_{L_x^2}|^2d\xi = \Vert u\Vert_{L^2}^2 < \infty.
	\end{equation*}
Thus, by the dominated convergence theorem we have
	\begin{equation}\begin{aligned}\label{reg3}
		&\lim_{\eta \to 0}\Vert u - (1 + \eta H)^{-m/2}(1 - \eta \partial^2_y)^{-m/2}u\Vert_{L^2}^2 \\
		&\, = \sum_{n=0}^\infty\int_{\mathbb{R}} \lim_{\eta \to 0}\left |(1 - (1 + \eta E_n)^{-m/2}(1 + \eta \xi^2)^{-m/2}) {\langle \hat{u}, \chi_n \rangle}_{L_x^2} \right |^2d\xi = 0.
	\end{aligned}\end{equation}
Applying (\ref{reg3}) to, respectively, $u_1 = \psi_0$, $u_2=H\psi_0,$ and $u_3=\partial^2_y\psi_0$ completes the proof of the lemma.
\end{proof}

%%%%%%%%%%%%%%%%%%%%%%%%%%%%%%%%%%%%%%%%%%%%%%
%%%%%%%%%%%%%%%%%%%%%%%%%%%%%%%%%%%%%%%%%%

\section{\textbf{Analysis of the filtered solution}}\label{sec:filter}

In this section, we shall rigorously analyze the asymptotic behavior of the filtered unknown $\phi^\eps=e^{itH/\varepsilon^2}\psi^\varepsilon$, as $\eps\to 0_+$. 
Recall that $\phi^\eps$ satisfies equation \eqref{filteredPDE} and thus, for all $t\in[0,T^{\eps}_{\rm max})$:
	\begin{equation}\begin{aligned}\label{filteredPDE2}
		\phi^\eps(t) &= \psi_0 + \frac{i}{2}\int_0^t\partial_{y}^2\phi^\eps(s)ds - \frac{b}{\varepsilon}\int_0^tG\left(\frac{s}{\varepsilon^2}, \partial_{y}\phi^\eps(s)\right)ds\\
		&\quad\quad\quad\quad- i\int_0^t\lambda F\left(\frac{s}{\varepsilon^2}, \phi^\eps(s)\right)ds.
	\end{aligned}\end{equation}
Our main task will be to show that
	\begin{equation}\label{task}
		-\frac{b}{\varepsilon}\int_0^tG\left(\frac{s}{\varepsilon^2}, \partial_{y}\phi^\eps(s)\right)ds = -\frac{i}{2}\int_0^t\partial^2_{y}\phi^\eps(s)ds + \sum_{j=1}^5G^\varepsilon_j(t),
	\end{equation}
where the $G^\varepsilon_j(t)$, $j=1, \dots, 5,$ will be defined below. In a second step, we shall  
prove that $G^\varepsilon_j(t) = \mathcal{O}(\eps)$ in the $L^2(\R^2)$-norm, uniformly on compact time-intervals.
\begin{remark}
Equation \eqref{task} may seem miraculous, and indeed we did not expect such a result at first. We already demonstrated that $G_{av}(u) = 0$, so clearly one must compute the higher order terms in $\eps$ in order to understand the convergence of the l.h.s. of \eqref{task}. It is then natural to use integration by parts to expand the integral, and after several substitutions we arrive at \eqref{task}.
\end{remark}

\subsection{Asymptotic expansion of the singular term}
We first note the differential identity
	\begin{equation}\label{ident}
		G\left(\frac{s}{\varepsilon^2}, \partial_{y}\phi^\eps(s)\right) = \varepsilon^2\partial_s\Big(\mathcal G\left(\frac{s}{\varepsilon^2}, \partial_{y}\phi^\eps(s)\right)\Big) - \varepsilon^2\mathcal G\left(\frac{s}{\varepsilon^2}, \partial_{y}\partial_s\phi^\eps(s)\right),
	\end{equation}
where here, and in the following, we denote
	\begin{equation}\label{Gtilde}
		\mathcal G\left(\theta, u\right) := \int_0^\theta G(\tau, u)\, d\tau.
	\end{equation}
Using \eqref{ident} to rewrite the third term on the r.h.s. of \eqref{filteredPDE2}, we find
	\begin{equation*}\begin{aligned}
		-\frac{b}{\varepsilon}\int_0^tG\left(\frac{s}{\varepsilon^2}, \partial_{y}\phi^\eps(s)\right)ds  = G^\varepsilon_1(t) + \varepsilon b\int_0^t\mathcal G\left(\frac{s}{\varepsilon^2}, \partial_{y}\partial_s\phi^\eps(s)\right)ds,
	\end{aligned}\end{equation*}
with 
	\begin{equation}\begin{split}\label{G1}
		G^\varepsilon_1(t) := -\varepsilon b\mathcal G\left(\frac{t}{\varepsilon^2}, \partial_{y}\phi^\eps(t)\right).
	\end{split}\end{equation}	
By substituting equation \eqref{filteredPDE} for $\partial_s\phi^\eps(s)$, we further see that
	\begin{equation}\begin{aligned}\label{GstepA}
		\varepsilon b\int_0^t\mathcal G\left(\frac{s}{\varepsilon^2}, \partial_{y}\partial_s\phi^\eps(s)\right)ds &= \varepsilon b\int_0^t\mathcal G\left(\frac{s}{\varepsilon^2}, \frac{1}{2}i\partial_{y}^3\phi^\eps(s)\right)ds\\
		&\quad\quad+ \varepsilon b\int_0^t\mathcal G\left(\frac{s}{\varepsilon^2}, -\frac{b}{\varepsilon}G\left(\frac{s}{\varepsilon^2}, \partial^2_{y}\phi^\eps(s)\right)\right)ds\\
		&\quad\quad+ \varepsilon b\int_0^t\mathcal G\left(\frac{s}{\varepsilon^2}, -i\partial_{y}\lambda F\left(\frac{s}{\varepsilon^2}, \phi^\eps(s)\right)\right)ds\\
		&= G_2^\varepsilon(t) + G_3^\varepsilon(t) - b^2\int_0^t\mathcal G\left(\frac{s}{\varepsilon^2}, G\left(\frac{s}{\varepsilon^2}, \partial^2_{y}\phi^\eps(s)\right)\right)ds,
	\end{aligned}\end{equation}
where we define
	\begin{equation}\begin{split}\label{G2}
		G^\varepsilon_2(t) := \varepsilon b\int_0^t\mathcal G\left(\frac{s}{\varepsilon^2}, \frac{1}{2}i\partial_{y}^3\phi^\eps(s)\right)ds
	\end{split}\end{equation}	
and
	\begin{equation}\begin{split}\label{G3}
		G^\varepsilon_3(t) := \varepsilon b\int_0^t\mathcal G\left(\frac{s}{\varepsilon^2}, -i\partial_{y}\lambda F\left(\frac{s}{\varepsilon^2}, \phi^\eps(s)\right)\right)ds.
	\end{split}\end{equation}	
Next, we shall compute the remaining term in \eqref{GstepA} by decomposing it over the normalized eigenfunctions of $H$,
\[
\chi_n(x) = \frac{1}{\sqrt{2^nn!}}\left(\frac{b}{\pi}\right)^{1/4}e^{-\frac{1}{2}bx^2}P_n(\sqrt{b}x), \quad n\in \N_0,
\]
where $P_n(z) = (-1)^ne^{z^2}\partial^n_z(e^{-z^2})$ are the Hermite polynomials.
Let
	\begin{equation*}
		u_n^\eps (t,y)= \langle \phi^\eps(t, \cdot, y), \chi_n\rangle_{L^2_x}, \quad v_{n,m} = \langle x\chi_n, \chi_m\rangle_{L^2_x}\in \C.
		\end{equation*}
Returning to the last term in \eqref{GstepA}, we can expand
	\begin{equation*}\begin{aligned}
		&\int_0^t\mathcal G\Big(\frac{s}{\varepsilon^2}, G\left(\frac{s}{\varepsilon^2}, \partial^2_{y}\phi^\eps(s)\Big)\right)ds\\
		&= \sum_n\sum_{m\neq n}\sum_{k\neq m}\int_0^t\int_0^{s/\varepsilon^2}e^{i\tau (E_k-E_m)}e^{is (E_m-E_n)/\varepsilon^2}v_{m,k}v_{n,m}\partial^2_{y}u_n^\eps(s)\chi_k \, d\tau ds,
	\end{aligned}\end{equation*}
and computing the integral in $\tau$ yields
	\begin{equation*}\begin{aligned}
		&-b^2\int_0^t\mathcal G\Big(\frac{s}{\varepsilon^2}, G\left(\frac{s}{\varepsilon^2}, \partial^2_{y}\phi^\eps(s)\Big)\right)ds\\
		&= G^\varepsilon_4(t) + ib^2\sum_n\sum_{m\neq n}\sum_{k\neq m}\frac{1}{E_k-E_m}\int_0^te^{is (E_k-E_n)/\varepsilon^2}v_{m,k}v_{n,m}\partial^2_{y}u_n^\eps(s)\chi_k\, ds,
	\end{aligned}\end{equation*}
where
	\begin{equation}\label{G4}
		G^\varepsilon_4(t) :=  -ib^2\sum_n\sum_{m\neq n}\sum_{k\neq m}\frac{1}{E_k-E_m}\int_0^te^{is (E_m-E_n)/\varepsilon^2}v_{m,k}v_{n,m}\partial^2_{y}u_n^\eps(s)\chi_k\, ds.
	\end{equation}	
The remaining term is then split over sums with $k= n$ and $k \neq n$, yielding
	\begin{equation}\begin{aligned}\label{GstepB}
		& ib^2\sum_n\sum_{m\neq n}\sum_{k\neq m}\frac{1}{E_k-E_m}\int_0^te^{is (E_k-E_n)/\varepsilon^2}v_{m,k}v_{n,m}\partial^2_{y}u_n^\eps(s)\chi_k\, ds \\
		&= G^\varepsilon_5(t) + ib^2\sum_n\sum_{m\neq n}\frac{1}{E_n-E_m}\int_0^t|v_{n,m}|^2\partial^2_{y}u_n^\eps(s)\chi_nds,
	\end{aligned}\end{equation}
where
	\begin{equation}\label{G5}
		G^\varepsilon_5(t) :=  ib^2\sum_n \sum_{m\neq n}\sum_{\substack{k\neq m\\ k\neq n}}\frac{1}{E_k-E_m}\int_0^te^{is (E_k-E_n)/\varepsilon^2}v_{m,k}v_{n,m}\partial^2_{y}u_n^\eps(s)\chi_k \, ds.
	\end{equation}	
We further simplify the second term on the r.h.s. of \eqref{GstepB} by using the orthogonality properties of the polynomials $P_n$ in $\chi_n$ (see e.g. \cite{messiah}). This yields
	\begin{equation*}\begin{aligned}
		v_{n,m} &= (2b)^{-1/2}\langle \sqrt{n+1}\chi_{n+1}, \chi_m \rangle_{L^2_x} + (2b)^{-1/2}\langle \sqrt{n}\chi_{n-1}, \chi_m\rangle_{L^2_x}\\
		&= (2b)^{-1/2}\left(\sqrt{n+1}\delta_{n+1,m} + \sqrt{n}\delta_{n-1,m}\right),
	\end{aligned}\end{equation*}
	where $\delta_{a, b}$ is the Kronecker delta.
Using this fact together with $E_n = b(n + \frac{1}{2})$, we obtain
	\begin{equation*}\begin{aligned}
		& ib^2\sum_n\sum_{m\neq n}\frac{1}{E_n-E_m}\int_0^t|v_{n,m}|^2\partial^2_{y}u_n^\eps(s)\chi_nds\\
		&= \frac{ib}{2}\sum_n\sum_{m\neq n}\frac{1}{E_n-E_m}\int_0^t(\sqrt{n+1}\delta_{n+1,m} + \sqrt{n}\delta_{n-1,m})^2\partial^2_{y}u_n^\eps(s)\chi_n \, ds\\
		&= \frac{ib}{2}\sum_n\frac{n+1}{-b}\int_0^t\partial^2_{y}u_n^\eps(s)\chi_nds + \frac{ib}{2}\sum_n\frac{n}{b}\int_0^t\partial^2_{y}u_n^\eps(s)\chi_n\, ds\\
		&= -\frac{i}{2}\sum_n\int_0^t\partial^2_{y}u_n^\eps(s)\chi_n\, ds \\
		&\equiv -\frac{1}{2}i\int_0^t\partial^2_{y}\phi^\eps(s)\, ds.
	\end{aligned}\end{equation*}
	
In summary, we obtain the announced identity \eqref{task}, in which the terms $G^\varepsilon_j(t)$, $j=1, \dots, 5,$ are given by, respectively, 
\eqref{G1}, \eqref{G2}, \eqref{G3}, \eqref{G4}, and \eqref{G5}. We consequently can rewrite \eqref{filteredPDE2} as
	\begin{equation}\begin{split}\label{filteredDuSimp}
		\phi^\eps(t) = \psi_0 + \sum_{j=1}^5G^\varepsilon_j(t) - i\int_0^t\lambda F\left(\frac{s}{\varepsilon^2}, \phi^\eps(s)\right)ds.
	\end{split}\end{equation}
Our next task is to derive appropriate bounds on the $G_j^\eps(t)$, $j=1, \dots, 5$. This will be done in several steps below.

%%%%%%% G^eps_j bounds

\subsection{Bounds on the expansion coefficients} First, we have a general bound on the function $\mathcal G(\theta, u)$ defined in \eqref{Gtilde}:

\begin{lemma}\label{tildeGbound}
For any $u\in \Sigma^{1}$,
\[
\sup_{\theta > 0}\Vert\mathcal G(\theta, u)\Vert_{L^2} \lesssim \Vert u\Vert_{\Sigma^{1}}.
\]
\end{lemma}
\begin{proof}
 As we have shown in (\ref{Gav}), $\mathcal G(\frac{2\pi}{b}, u) = 0$. Since $G(\tau, u)$ is $\frac{2\pi}{b}$-periodic in $\tau$, we then have that $\mathcal G(\theta, u)$ is $\frac{2\pi}{b}$-periodic in $\theta$. Thus,
	\begin{equation}\begin{aligned}
		\sup_{\theta > 0}\Vert\mathcal G(\theta, u)\Vert_{L^2} &= \sup_{\theta \in (0,2\pi/b]}\Vert\mathcal G(\theta, u)\Vert_{L^2}\\
		&\leq \sup_{\theta \in (0,2\pi/b]}\int_0^\theta\Vert e^{i\tau H}xe^{-i\tau H}u\Vert_{L^2}d\tau\\
		&= \int_0^{2\pi/b}\Vert e^{i\tau H}xe^{-i\tau H}u\Vert_{L^2}d\tau.
	\end{aligned}\end{equation}
Using Lemmas \ref{1} and \ref{Ubound}, we find
	\begin{equation}
		\sup_{\theta > 0}\Vert\mathcal G(\theta, u)\Vert_{L^2} \lesssim \int_0^{2\pi/b}\Vert u\Vert_{\Sigma^1}d\tau
		\lesssim \Vert u\Vert_{\Sigma^1}.
	\end{equation}
\end{proof}
This general bound on $\mathcal G$ then directly yields the following bounds on $G_1^\eps(t)$, $G_2^\eps(t)$, and $G_3^\eps(t)$:
\begin{lemma}\label{GepsLem1} It holds
	\begin{equation}\label{Geps12}
		\|G_1^\eps(t)\|_{L^2} \lesssim \eps\Big(\sup_{s\in[0,t]}\|\phi^\eps(s)\|_{\Sigma^2}\Big), \quad \|G_2^\eps(t)\|_{L^2} \lesssim \eps t\Big(\sup_{s\in[0,t]}\|\phi^\eps(s)\|_{\Sigma^4}\Big)
	\end{equation}
as well as
	\begin{equation}\label{Geps3}
		\|G_3^\eps(t)\|_{L^2} \lesssim \eps t|\lambda|\Big(\sup_{s\in[0,t]}\|\phi^\eps(s)\|_{\Sigma^2}^{2\sigma+1}\Big).
	\end{equation}
\end{lemma}
\begin{proof}
The estimates \eqref{Geps12} follow by the definition of $G_1^\eps(t)$ and $G_2^\eps(t)$ and the previous lemma. Moreover, \eqref{Geps3} can be 
bounded in the same way, using the additional fact that
\[
\left \|\lambda F\left(\frac{s}{\varepsilon^2}, \phi^\eps(s)\right)\right \|_{\Sigma^2} \lesssim |\lambda|\cdot\|\phi^\eps(s)\|_{\Sigma^2}^{2\sigma+1},
\]
which follows from Lemma \ref{Ubound} and Corollary \ref{MoserCor}.
\end{proof}
The bounds on $G^\eps_4(t)$ and $G^\eps_5(t)$ are more complicated, and they are the main reason for our strong regularity requirements on $\phi$:
\begin{proposition}\label{GepsLem2} We have that
	\begin{equation}\label{Geps4}
		\|G_4^\eps(t)\|_{L^2} \lesssim \eps(1+t)\sup_{s\in[0,t]}\left(\|\phi^{\eps}(s)\|_{\Sigma^{15}}+\|\phi^{\eps}(s)\|_{\Sigma^{15}}^{2\sigma+1}\right)
	\end{equation}
and
	\begin{equation}\label{Geps5}
		\|G_5^\eps(t)\|_{L^2} \lesssim \eps(1+t)\sup_{s\in[0,t]}\left(\|\phi^{\eps}(s)\|_{\Sigma^{15}}+\|\phi^{\eps}(s)\|_{\Sigma^{15}}^{2\sigma+1}\right).
	\end{equation}
\end{proposition}
\begin{proof}
We begin with \eqref{Geps4}. Note that the integral term in $G_4^\eps(t)$ can be rewritten as
	\begin{equation*}\begin{aligned}
		& \int_0^t e^{is (E_m-E_n)/\varepsilon^2}v_{m,k}v_{n,m}\partial^2_{y}u_n^\eps(s)\chi_k\, ds\\
		&= \int_0^t\Big(\frac{-i\varepsilon^2}{E_m-E_n}\, \partial_se^{is (E_m-E_n)/\varepsilon^2}\Big)v_{m,k}v_{n,m}\partial^2_{y}u_n^\eps(s)\chi_k\, ds,
	\end{aligned}\end{equation*}
and integrating by parts yields
	\begin{equation*}\begin{aligned}
		& \int_0^te^{is (E_m-E_n)/\varepsilon^2}v_{m,k}v_{n,m}\partial^2_{y}u_n^\eps(s)\chi_kds\\
		&= \int_0^t\frac{i\varepsilon^2}{E_m-E_n}e^{is (E_m-E_n)/\varepsilon^2}v_{m,k}v_{n,m}\partial^2_{y}\partial_su_n^\eps(s)\chi_k\, ds\\
		&\quad\quad\quad+ \frac{\varepsilon^2}{i(E_m-E_n)}e^{it (E_m-E_n)/\varepsilon^2}v_{m,k}v_{n,m}\partial^2_{y}u_n^\eps(t)\chi_k\\
		&\quad\quad\quad- \frac{\varepsilon^2}{i(E_m-E_n)}v_{m,k}v_{n,m}\partial^2_{y}u_n^\eps(0)\chi_k.
	\end{aligned}\end{equation*}
Plugging this into \eqref{G4}, we get
	\begin{equation}\begin{aligned}\label{G4step}
		\|G_4^\eps(t)\|_{L^2} &\lesssim \eps^2\sum_n\sum_{m\neq n}\sum_{k\neq m}\Big(\int_0^t|v_{m,k}|\, |v_{n,m}|\, \|\partial^2_{y}\partial_su_n^\eps(s)\chi_k\|_{L^2}ds\\
		&\qquad\qquad\qquad\qquad + |v_{m,k}|\, |v_{n,m}|\, \|\partial^2_{y}\big(u_n^\eps(t)+u_n^\eps(0)\big)\chi_k\|_{L^2}\Big).
	\end{aligned}\end{equation}
To proceed, we prove two useful estimates. First, for any $p \in \N_0$ we have
	\begin{equation*}\begin{aligned}
		E_m^{p/2}|v_{n,m}| &= \left |\langle  H^{p/2}(x\chi_n), \chi_m \rangle_{L^2_x} \right|\\
		&\lesssim \| H^{p/2}(\sqrt{n+1}\chi_{n+1} + \sqrt{n}\chi_{n-1})\|_{L^2_x} \lesssim E_{n}^{(p+1)/2},
	\end{aligned}\end{equation*}
and so
	\begin{equation}\begin{split}\label{vEst}
		|v_{n,m}| \lesssim \frac{E_{n}^{(p+1)/2}}{E_m^{p/2}}.
	\end{split}\end{equation}
Second, let $f_n \equiv \langle f, \chi_n\rangle_{L^2_x}$, where $f=f(x,y)$ is some sufficiently regular function. Then, for any $q\in \N_0$,
	\begin{equation*}\begin{aligned}
		E_n^q\| f_n\chi_k\|_{L^2}^2 &= E_n^q\| f_n\|_{L^2_y}^2 = \int \left|\langle f, H^{q/2}\chi_n\rangle_{L^2_x}\right|^2 dy\\
		&\leq \int\|H^{q/2}f\|_{L^2_x}^2 \, dy \lesssim \|f\|_{\Sigma^{q}}^2,
	\end{aligned}\end{equation*}
where the last line follows from Lemma \ref{1}. Thus,
	\begin{equation}\label{fEst}
		\| f_n\chi_k\|_{L^2} \lesssim E_n^{-q/2}\|f\|_{\Sigma^{q}}.
	\end{equation}
Using \eqref{vEst} and \eqref{fEst}, we can estimate	
\begin{equation*}\begin{aligned}
		\sum_n\sum_{m\neq n}&\sum_{k\neq m}|v_{m,k}|\, |v_{n,m}|\, \|f_n\chi_k\|_{L^2}\\
		&\lesssim \sum_n\sum_{m}\sum_{k}\frac{E_m^{(p+1)/2}}{E_k^{p/2}}\cdot\frac{E_n^{(l+1)/2}}{E_m^{l/2}}\cdot E_n^{-q/2}\|f\|_{\Sigma^{q}}\\
		&= \sum_n\sum_{m}\sum_{k}E_k^{-3/2}E_m^{-3/2}E_n^{-3/2}\|f\|_{\Sigma^{11}} \lesssim \|f\|_{\Sigma^{11}},
	\end{aligned}\end{equation*}
where we chose $p = 3$, $l = 7$, $q = 11$, and recalling that $E_n = b(n + \frac{1}{2})$. Applying this to \eqref{G4step}, we have
	\begin{equation*}\begin{aligned}
		\|G_4^\eps(t)\|_{L^2} &\lesssim \eps^2\Big(\int_0^t\|\partial^2_{y}\partial_s\phi^{\eps}(s)\|_{\Sigma^{11}}ds + \|\partial^2_{y}\big(\phi^{\eps}(t)+\phi^{\eps}(0)\big)\|_{\Sigma^{11}}\Big)\\
		&\lesssim \eps^2\Big(\int_0^t\|\partial_s\phi^{\eps}(s)\|_{\Sigma^{13}}ds + \|\phi^{\eps}(t)+\phi^{\eps}(0)\|_{\Sigma^{13}}\Big).
	\end{aligned}\end{equation*}
As previously done in the estimate \eqref{wpEst1}, one can see from the filtered equation \eqref{filteredPDE}, that
	\begin{equation}\begin{split}
		\|\partial_s\phi^{\eps}(s)\|_{\Sigma^{13}} \lesssim \frac{1}{\eps}\left(\|\phi^{\eps}(s)\|_{\Sigma^{15}}+\|\phi^{\eps}(s)\|_{\Sigma^{15}}^{2\sigma+1}\right),
	\end{split}\end{equation}
and thus, we obtain the desired estimate \eqref{Geps4}. The proof for \eqref{Geps5} is nearly identical, and hence we omit the details.
\end{proof}

%%%%%%%%%%%%%%%%%%%%%%%%%%%%%%%%%%%%%%%%%
%%%%%%%%%%%%%%%%%%%%%%%%%%%%%%%%%%%%%%%%%

\section{\textbf{Proof of the asymptotic approximation result}}\label{sec:stab}

This section is devoted to the proof of item (ii) in Theorem \ref{main}. We begin by establishing a time interval on which certain estimates can be controlled.
Let $\phi$ and $\psi^{\varepsilon}$ be the solutions to (\ref{limeq}) and (\ref{pde2}), respectively, subject to the same initial data $\psi_0\in \Sigma^{15}$. 
Let $T_{\rm max}$ and $T^{\varepsilon}_{\rm max}$ be the maximal existence times for $\phi$ and $\psi^{\varepsilon}$, respectively, 
as given in Propositions \ref{wp2} and \ref{wp1}. Fix $T \in (0,T_{\rm max})$ and set
\begin{equation*}
		M = \sup_{\varepsilon > 0}\Vert e^{-itH/\varepsilon^2}\phi\Vert_{L^\infty((0,T)\times\mathbb{R}^2)}.
\end{equation*}
Since $\Sigma^{15} \hookrightarrow L^\infty$, Lemma \ref{Ubound} implies that
	\begin{equation*}
		\Vert e^{-itH/\varepsilon^2}\phi\Vert_{L^\infty((0,T)\times\mathbb{R}^2)} \lesssim \Vert e^{-itH/\varepsilon^2}\phi\Vert_{L^\infty((0,T), \Sigma^{15})} = \Vert \phi\Vert_{L^\infty((0,T), \Sigma^{15})}.
	\end{equation*}
Therefore,
	\begin{equation}\label{Mcon}
		\Vert \psi_0\Vert_{L^\infty} = \Vert \phi(0)\Vert_{L^\infty} \leq M \lesssim \Vert \phi\Vert_{L^\infty((0,T), \Sigma^{15})} < \infty.
	\end{equation}
Let us define
	\begin{equation*}
		T^{\eps} = \sup \Big\{ t \in [0,T^{\eps}_{\rm max})\, : \, \Vert \psi^{\varepsilon}(s)\Vert_{L^\infty} \leq 2M \,\,\forall s\in[0,t]\Big\}.
	\end{equation*}
By \eqref{Mcon} and the continuity of $\psi^{\varepsilon}$, we have $T^{\eps} > 0$. 
Recalling Duhamel's formula \eqref{duhamel}, we can now use Lemma \ref{Ubound} and Lemma \ref{Moser} to infer that, for all $t \in [0, T^{\eps})$: 
	\begin{equation*}\begin{split}
		\Vert\psi^{\varepsilon}(t)\Vert_{\Sigma^{15}} &\leq \Vert e^{-itH_\varepsilon/\varepsilon^2}\psi_0\Vert_{\Sigma^{15}} + \int_0^t\Vert e^{-i(t-s)H_\varepsilon/\varepsilon^2}\lambda|\psi^{\varepsilon}(s)|^{2\sigma}\psi^{\varepsilon}(s)\Vert_{\Sigma^{15}}ds\\
		&\lesssim \Vert\psi_0\Vert_{\Sigma^{15}} + \int_0^t|\lambda|\Vert \psi^{\varepsilon}(s)\Vert_{L^\infty}^{2\sigma}\Vert\psi^{\varepsilon}(s)\Vert_{\Sigma^{15}}ds\\
		&\leq \Vert\psi_0\Vert_{\Sigma^{15}} + \int_0^t|\lambda|(2M)^{2\sigma}\Vert\psi^{\varepsilon}(s)\Vert_{\Sigma^{15}}ds.
	\end{split}\end{equation*}
Using Gr\"onwall's inequality, we consequently obtain
	\begin{equation}\label{Sigma4bound}
		\Vert\psi^{\varepsilon}(t)\Vert_{\Sigma^{15}} \lesssim \Vert\psi_0\Vert_{\Sigma^{15}}e^{t(2M)^{2\sigma}\, |\lambda|}, \quad \text{for all $t \in [0, T^{\eps})$.}
	\end{equation}
A consequence of this bound is that if $T^{\eps} < \infty$, then
	\begin{equation}\label{Tcon}
		T^{\eps} < T^{\eps}_{\rm max} \,\,\,\text{and}\,\,\, \Vert \psi^{\varepsilon}(T^{\eps})\Vert_{L^\infty} = 2M.
	\end{equation}
To see this, we note that if $T^{\eps} = T^{\eps}_{\rm max} < \infty$, then 
\[
\lim_{t \to T^{\eps}_{\rm max}}\Vert\psi^{\varepsilon}(t)\Vert_{\Sigma^{15}} < \infty,
\] 
which contradicts \eqref{blowup} in Proposition \ref{wp1}. Moreover, if $\Vert \psi^{\varepsilon}(T^{\eps})\Vert_{L^\infty} < 2M$, then, by continuity, 
$\Vert \psi^{\varepsilon}(T^{\eps} + \delta)\Vert_{L^\infty} < 2M$ for some $\delta > 0$, which contradicts the definition of $T^{\eps}$ above.

%%%%%%%%%%%%%%%%%%%%%%%%%%%%%%%%%%%%%%%
%%%%%%%%%%%%%%%%%%%%%%%%%%%%%%%%%%%%%%%%%%
% Estimate of approximation

\subsection{Nonlinear stability of the approximation}

 We are now ready to estimate the difference $\Vert\phi^\eps - \phi\Vert_{L^2}$, in the limit $\eps \to 0_+$. 
 Recalling the Duhamel formulations \eqref{limeqDu} and \eqref{filteredDuSimp}, we have, for all $t \in [0,\min\{T, T^{\eps}\}]$:
	\begin{equation}\begin{split}\label{uForm}
		\phi^\eps (t)- \phi (t)&= \sum_{j=1}^5G^\varepsilon_j(t) -i\lambda\int_0^t\left(F(\frac{s}{\varepsilon^2}, \phi^\eps(s)) - F_{\rm av}(\phi(s))\right)ds\\
		&= \sum_{j=1}^5G^\varepsilon_j(t) -i\lambda\int_0^t\left(F(\frac{s}{\varepsilon^2}, \phi^\eps(s)) - F(\frac{s}{\varepsilon^2}, \phi(s))\right)ds\\
		&\quad\quad\quad\quad - i\lambda\int_0^t\left(F(\frac{s}{\varepsilon^2}, \phi(s)) - F_{\rm av}(\phi(s))\right)ds\\
		&:= \sum_{j=1}^5G^\varepsilon_j (t)- i\lambda A_1 (t) - i\lambda A_2 (t).
	\end{split}\end{equation}
This allows us to estimate
	\begin{equation}\begin{aligned}\label{uFormBound}
		\|\phi^\eps (t)- \phi (t)\|_{L^2} &\leq \sum_{j=1}^5\|G^\varepsilon_j(t)\|_{L^2} + |\lambda|\big( \|A_1(t)\|_{L^2} +  \|A_2(t)\|_{L^2}\big),
	\end{aligned}\end{equation}
and we can bound the first term on the r.h.s. using the following lemma:
\begin{lemma}\label{GepsCor}
For all $t \in [0,\min\{T, T^{\eps}\}]$, it holds
	\begin{equation}\begin{split}
		\sum_{j=1}^5\Vert G^\varepsilon_j(t)\Vert_{L^2} \leq \varepsilon C(T).
	\end{split}\end{equation}
\end{lemma}
\begin{proof}
From \eqref{Sigma4bound} and Lemma \ref{Ubound}, we have
	\begin{equation}\begin{split}\label{phiBound}
		\Vert \phi^{\eps}(t)\Vert_{\Sigma^{15}} \lesssim \Vert \psi^{\eps}(t)\Vert_{\Sigma^{15}} \lesssim \|\psi_0\|_{\Sigma^{15}}e^{T(2M)^{2\sigma}|\lambda|}
	\end{split}\end{equation}
for all $t \in [0,\min\{T, T^{\eps}\}]$. Applying this estimate to Lemmas \ref{GepsLem1} and \ref{GepsLem2}, and bounding $t$ above by $T$, completes the proof.
\end{proof}
It therefore remains to bound $\|A_1(t)\|_{L^2}$ and $\|A_2(t)\|_{L^2}$. To this end, we first note that
\[
 \Vert A_1 (t)\Vert_{L^2} \leq \int_0^t \Big\Vert |e^{-is H/\varepsilon^2}\phi^\eps(s)|^{2\sigma}e^{-is H/\varepsilon^2}\phi^\eps(s) - |e^{-is H/\varepsilon^2}\phi(s)|^{2\sigma}e^{-is H/\varepsilon^2}\phi(s)\Big\Vert_{L^2}ds.
\]
This can be estimated further by using well-known (local) Lipschitz estimate: for $u, v \in L^2\cap L^\infty$ and $\sigma \in \mathbb{N}$,
\begin{equation}\label{LipEst}
	\Big\Vert |u|^{2\sigma}u - |v|^{2\sigma}v\Big\Vert_{L^2} \lesssim (\Vert u\Vert_{L^\infty}^{2\sigma} + \Vert v\Vert_{L^\infty}^{2\sigma})\Vert u - v\Vert_{L^2}.
\end{equation}
Since $\Sigma^{15} \hookrightarrow L^\infty$, we obtain
	\begin{equation}\begin{split}\label{A1bound}
		\Vert A_1 (t)\Vert_{L^2} 
		& \lesssim \int_0^t \left(\Vert e^{-is H/\varepsilon^2}\phi^\eps(s)\Vert_{L^\infty}^{2\sigma} + \Vert e^{-is H/\varepsilon^2}\phi(s)\Vert_{L^\infty}^{2\sigma}\right)\Vert \phi^\eps - \phi\Vert_{L^2}ds\\
		&\lesssim (2M)^{2\sigma}\int_0^t \Vert \phi^\eps(s) - \phi(s)\Vert_{L^2}ds.
	\end{split}\end{equation}
To obtain a bound for $\|A_2(t)\|_{L^2}$, we use the following differential identity
	\begin{equation}\begin{split}\label{Fdiff}
		F\left(\frac{s}{\varepsilon^2}, \phi(s)\right) - F_{\rm av}\left(\phi(s)\right) &= \varepsilon^2\partial_s\left(\int_0^{s/\eps^2} \Big(F(\tau, \phi(s)) - F_{\rm av}(\phi(s)) \Big)d\tau\right)\\
		&\quad - \varepsilon^2\int_0^{s/\eps^2} \partial_s\Big(F(\tau, \phi(s)) - F_{\rm av}(\phi(s)) \Big)d\tau,
	\end{split}\end{equation}
which yields
	\begin{equation}\begin{split}\label{A2step}
		A_2 (t)&= \varepsilon^2\int_0^{t/\eps^2} \Big(F(\tau, \phi(t)) - F_{\rm av}(\phi(t)) \Big)d\tau\\
		&\quad\quad - \varepsilon^2\int_0^t\int_0^{s/\eps^2} \partial_s\Big(F(\tau, \phi(s)) - F_{\rm av}(\phi(s)) \Big)d\tau ds.
	\end{split}\end{equation}
Recalling the definition of $F_{\rm av}$ \eqref{eq:average}, we clearly have
	\begin{equation*}\begin{split}
		\int_0^{2\pi/b} \Big(F(\tau, \phi(t)) - F_{\rm av}(\phi(t)) \Big)d\tau = 0.
	\end{split}\end{equation*}
Then, since $F(\tau, \phi(t))$ is $\frac{2\pi}{b}$-periodic in $\tau$, we obtain
	\begin{equation*}
		\sup_{\theta > 0}\Big|\int_0^{\theta} \Big(F(\tau, \phi(t)) - F_{\rm av}(\phi(t)) \Big)d\tau\Big| = \sup_{\theta \in [0,\frac{2\pi}{b}]}\Big|\int_0^{\theta} \Big(F(\tau, \phi(t)) - F_{\rm av}(\phi(t)) \Big)d\tau\Big|.
	\end{equation*}
Similarly, one can show that
	\begin{equation*}
		\sup_{\theta > 0}\Big|\int_0^{\theta} \partial_s\Big(F(\tau, \phi(s)) - F_{\rm av}(\phi(s)) \Big)d\tau\Big| = \sup_{\theta \in [0,\frac{2\pi}{b}]}\Big|\int_0^{\theta} \partial_s\Big(F(\tau, \phi(s)) - F_{\rm av}(\phi(s)) \Big)d\tau\Big|.
	\end{equation*}
Applying these to \eqref{A2step}, we obtain
	\begin{equation}\begin{split}\label{A2step2}
		\|A_2(t)\|_{L^2} &\leq \varepsilon^2\int_0^{2\pi/b} \Big\|F(\tau, \phi(t)) - F_{\rm av}(\phi(t)) \Big\|_{L^2}d\tau\\
		&\quad\quad + \varepsilon^2\int_0^t\int_0^{2\pi/b} \Big\|\partial_s\Big(F(\tau, \phi(s)) - F_{\rm av}(\phi(s)) \Big)\Big\|_{L^2}d\tau ds.
	\end{split}\end{equation}
We can now use Lemma \ref{Ubound} and Corollary \ref{MoserCor} to estimate
	\begin{equation*}
		\Big\|F(\tau, \phi(t)) - F_{\rm av}(\phi(t)) \Big\|_{L^2} \lesssim \Big\|F(\tau, \phi(t)) - F_{\rm av}(\phi(t)) \Big\|_{\Sigma^{15}}\lesssim \|\phi(t)\|_{\Sigma^{15}}^{2\sigma+1},
		\end{equation*}
and, similarly, we obtain
	\begin{equation*}
		\Big\|\partial_s\Big(F(\tau, \phi(s)) - F_{\rm av}(\phi(s)) \Big)\Big\|_{L^2} \lesssim \|\phi(s)\|_{\Sigma^{15}}^{2\sigma}\|\partial_s\phi(s)\|_{\Sigma^{15}}\lesssim \|\phi(s)\|_{\Sigma^{15}}^{4\sigma+1}.
	\end{equation*}
Combining these estimates with \eqref{A2step2} yields
	\begin{equation*}\begin{split}
		\|A_2(t)\|_{L^2} &\lesssim \eps^2(1+t)\sup_{s\in[0,t]}\left(\|\phi(s)\|_{\Sigma^{15}}^{2\sigma+1}+\|\phi(s)\|_{\Sigma^{15}}^{4\sigma+1}\right).
	\end{split}\end{equation*}
Finally, recalling \eqref{phiBound}, we can bound
	\begin{equation}\begin{split}\label{A2bound}
		\|A_2(t)\|_{L^2} &\lesssim \eps^2C(T)
	\end{split}\end{equation}
for all $t\in [0,\min\{T,T^{\eps}\}]$.

\begin{proof}[\textbf{Proof of item (ii) of Theorem \ref{main}}]We can now apply the bounds given by \eqref{A1bound}, \eqref{A2bound}, and Lemma \ref{GepsCor} to \eqref{uFormBound}, to obtain
	\begin{equation}\begin{split}
		\Vert \phi^\eps(t) - \phi(t)\Vert_{L^2} \lesssim C_1\varepsilon + C_2\int_0^t \Vert \phi^\eps (s)- \phi(s)\Vert_{L^2}ds
	\end{split}\end{equation}
for all $t\in [0,\min\{T,T^{\eps}\}]$, where $C_1$ and $C_2$ may depend on $T$ but not $\eps$.
Thus, by Gr\"onwall's inequality we have
	\begin{equation}\begin{split}\label{uBound}
		\Vert \psi^{\varepsilon} (t)- e^{-itH/\varepsilon^2}\phi(t)\Vert_{L^2} = \|\phi^\eps(t) - \phi(t)\Vert_{L^2} \lesssim C_1\varepsilon e^{C_2t} \leq C_1\varepsilon e^{C_2T}
	\end{split}\end{equation}
for all $t\in [0,\min\{T,T^{\eps}\}]$. To complete the proof of Theorem \ref{main}, it remains to show that $\min\{T,T^{\eps}\} = T$. 

First, we use the Gagliardo-Nirenberg inequality, together with (\ref{uBound}) and Lemma \ref{Ubound}, to obtain
	\begin{equation}\begin{aligned}\label{RevEq1}
		 \Vert \psi^{\varepsilon}(t)\Vert_{L^\infty} &\leq \|e^{-itH/\varepsilon^2}\phi(t)\Vert_{L^\infty} + \Vert \psi^{\varepsilon} (t)- e^{-itH/\varepsilon^2}\phi(t)\Vert_{L^\infty}\\
		&\leq M + C'\Vert \psi^{\varepsilon} (t)- e^{-itH/\varepsilon^2}\phi(t)\Vert_{L^2}^{1/2}\, \Vert \psi^{\varepsilon} (t)- e^{-itH/\varepsilon^2}\phi(t)\Vert_{H^2}^{1/2}\\
		&\leq M + C''\varepsilon^{1/2}(\Vert \psi^{\varepsilon}(t)\Vert_{\Sigma^{15}} + \Vert \phi(t)\Vert_{\Sigma^{15}})^{1/2},
\end{aligned}\end{equation}
where $C''$ is independent of $\varepsilon$. Recalling the estimate \eqref{Sigma4bound}, we have the $\eps$-independent bound
\[
	\Vert\psi^{\varepsilon}(t)\Vert_{\Sigma^{15}} \lesssim \Vert\psi_0\Vert_{\Sigma^{15}}e^{T(2M)^{2\sigma}\, |\lambda|}
\]
for all $t \in [0,\min\{T,T^{\eps}\}]$. Additionally, since $T < T_{\rm max}$, $\Vert \phi(t)\Vert_{\Sigma^{15}}$ can be bounded uniformly on $[0,\min\{T,T^{\eps}\}]$. Applying these estimates to \eqref{RevEq1} yields
	\begin{equation}\begin{split}
		 \Vert \psi^{\varepsilon}(t)\Vert_{L^\infty} \leq M + C'''\varepsilon^{1/2},
	\end{split}\end{equation}
where $C'''$ is again independent of $\varepsilon$. Choosing $\varepsilon_T = (M/C''')^2$ and restricting $\varepsilon \in (0, \varepsilon_T)$, we get
	\begin{equation}\begin{split}\label{LinfBound}
		 \Vert \psi^{\varepsilon}(t)\Vert_{L^\infty} < 2M,
	\end{split}\end{equation}
for all $t \in [0,\min\{T,T^{\eps}\}]$. Consequently, we must have $T < T^{\eps}$. This is clear if $T^{\eps} = \infty$. If $T^{\eps} < \infty$, then (\ref{Tcon}) contradicts (\ref{LinfBound}) unless $\min\{T^{\eps},T\} = T < T^{\eps}$.
In summary, we have shown
	\begin{equation*}
		\Vert \psi^{\varepsilon} (t)- e^{-itH/\varepsilon^2}\phi (t)\Vert_{L^2} \leq C_T\varepsilon
	\end{equation*}
for all $t \in [0,T]$ and $\eps \in (0, \varepsilon_T)$, which completes the proof of Theorem \ref{main}.\end{proof}

%%%%%%%%%%%%%%%%%%%%%%%%%%%%%%%%%%%%%%%%%%%%%%%%%%%%%%%%%%

\section{\textbf{Proof of Theorem \ref{mainCor}}}\label{sec:stabCor}

The proof of our second theorem is similar to the one for the previous result, and thus we will only discuss the main differences below.

A first difficulty arises from the fact that $\phi^\eps(t, \cdot) \in \Sigma^2$, only. Hence, we cannot directly quote Lemma \ref{GepsCor} to bound the terms $G^\eps_j = G^\eps_j(t,\phi^\eps)$, $j=1, \dots, 5$. 
To overcome this obstacle, we approximate $\phi^\eps$ with a regularized solution.
For any $\eta > 0$, let $\psi_0^\eta \in \Sigma^{15}$ be the regularization of $\psi_0$ as defined in Lemma \ref{regLem}. Let $\phi^{\eps,\eta}(t, \cdot)$ 
be the solution to \eqref{filteredPDE} with initial data $\psi_0^\eta\in \Sigma^{15}$. By the triangle inequality,
	\begin{equation*}
		\Vert \phi^\eps (t) - \phi (t) \Vert_{L^2} \leq \Vert \phi^\eps (t) - \phi^{\eps,\eta} (t) \Vert_{L^2} + \Vert \phi^{\eps,\eta} (t) - \phi (t) \Vert_{L^2}.
	\end{equation*}
Here, the first term on the r.h.s. can be bounded using the continuous dependence on initial data result for sufficiently small $\eta$. 
For the other term, we have, similarly to the bound \eqref{uFormBound}, that
	\begin{equation}\begin{aligned}\label{uBound2}
		\|\phi^{\eps,\eta} (t)- \phi (t)\|_{L^2} \leq & \, \|\psi_0^\eta-\psi_0\|_{L^2} + \sum_{j=1}^5\|G^\varepsilon_j(t,\phi^{\eps,\eta})\|_{L^2} \\
		&+ \Big \|\int_0^t\widetilde{A_1}(s) \,ds  \Big\|_{L^2} + \Big \|\int_0^t\widetilde{A_2}(s) \, ds\Big\|_{L^2},
	\end{aligned}\end{equation}
where we define
	\begin{equation}\begin{split}
		\widetilde{A_1} (s)= e^{isH/\eps^2}\lambda(\eps x, y)\Big(&|e^{-isH/\eps^2}\phi^{\eps,\eta}(s)|^{2\sigma}e^{-isH/\eps^2}\phi^{\eps,\eta}(s)\\
		&\, - |e^{-isH/\eps^2}\phi(s)|^{2\sigma}e^{-isH/\eps^2}\phi(s)\Big)
	\end{split}\end{equation}
and
	\begin{equation}\begin{split}
		\widetilde{A_2} (s)= e^{isH/\eps^2}\lambda(\eps x, y)|e^{-isH/\eps^2}\phi(s)|^{2\sigma}e^{-isH/\eps^2}\phi(s) - \lambda(0,y)F_{\rm av}(\phi(s)).
	\end{split}\end{equation}
Using the same estimates as in Lemmas \ref{GepsCor} and \ref{regLem} then yields for $j=1, \dots, 5$:
\[
	\|G^\varepsilon_j(t,\phi^{\eps,\eta})\|_{L^2} \leq \eps C(T, \|\psi_0^\eta\|_{\Sigma^{15}}) \leq \eps C(T, (1+\eta^{-13/2})\|\psi_0\|_{\Sigma^{2}}).
\]
Note, however, that this bound grows as $\eta \to 0$, and so we must first fix $\eta_\delta$ 
sufficiently small so that $\|\psi_0^\eta-\psi_0\|_{L^2} \leq \delta$ for some arbitrary $\delta > 0$, and then choose $\eps_\delta$ 
such that $\eps_\delta C(T, (1+\eta_\delta^{-13/2})\|\psi_0\|_{\Sigma^{2}}) \leq \delta$. The price to pay for our regularization is thus the  
loss of the $\mathcal{O}(\eps)$-convergence rate, and instead we only obtain the limiting statement \eqref{mainCorCon}.

We can now try to estimate the remaining terms in \eqref{uBound2} by following the same ideas as for the corresponding terms in \eqref{uFormBound}. The estimate on $\widetilde{A_1}$ follows similarly to $A_1$ in \eqref{A1bound}, where the only additional step is to bound $\lambda$ in the $L^\infty$-norm. To estimate $\widetilde{A_2}$, the main 
obstacle stems from the fact that, for any $\eps>0$, the function $\lambda(\eps \cdot, \cdot) \in W^{15,\infty}(\R^2)$ and the operator $H$ no longer commute. First, we define
	\begin{equation*}\begin{split}
		\widetilde{F}(\theta,u) := \int_0^\theta \left(e^{i\tau H}\lambda(\eps x, y)|e^{-i\tau H}u|^{2\sigma}e^{-i\tau H}u - \lambda(0,y)F_{\rm av}(u)\right)d\tau.
	\end{split}\end{equation*}
This yields, analogously to \eqref{A2step},
	\begin{equation}\begin{aligned}\label{A2tilde}
		& \int_0^t\widetilde{A_2}(s) \, ds = \eps^2\widetilde{F}\left(\frac{t}{\varepsilon^2}, \phi(t)\right) \\
		&- \eps^2\int_0^t\int_0^{s/\eps^2} \partial_s\Big(e^{i\tau H}\lambda(\eps \cdot) |e^{-i\tau H}\phi(s)|^{2\sigma}e^{-i\tau H}\phi(s)
		- \lambda(0,\cdot)F_{\rm av}(\phi(s))\Big)d\tau ds.
	\end{aligned}\end{equation}
We present the bound for $\widetilde{F}$ in the following lemma.
\begin{lemma}For all $t \in [0, T_{max})$,
	\[
		\Big\Vert\widetilde{F}\left(\frac{t}{\varepsilon^2}, \phi(t)\right)\Big\Vert_{L^2} \lesssim \frac{(1+t)}{\eps}\left(\Vert \phi(t)\Vert_{\Sigma^{2}} + \Vert \phi(t)\Vert_{\Sigma^{2}}^{2\sigma+1}\right).
	\]
\end{lemma}
\begin{proof}
In the spirit of Lemma \ref{tildeGbound}, we will decompose $\widetilde{F}(\theta, u)$ into two parts, the first of which is $\frac{2\pi}{b}$-periodic in $\theta$ with $0$-average:
	\begin{equation*}\begin{aligned}
		\widetilde{F}(\theta, u) &= \int_0^\theta \Big(e^{i\tau H}\lambda(\eps x, y)|e^{-i\tau H}u|^{2\sigma}e^{-i\tau H}u\\
		&\quad\quad\quad\quad\quad\quad - \frac{b}{2\pi}\int_0^{\frac{2\pi}{b}}e^{i\rho H}\lambda(\varepsilon x, y)|e^{-i\rho H}u|^{2\sigma}e^{-i\rho H}ud\rho\Big)d\tau\\
		&\quad\quad + \int_0^\theta \left(\frac{b}{2\pi}\int_0^{\frac{2\pi}{b}}e^{i\rho H}\lambda(\varepsilon x, y)|e^{-i\rho H}u|^{2\sigma}e^{-i\rho H}ud\rho - \lambda(0,y)F_{\rm av}(u)\right)d\tau\\
		&:= F_1(\theta, u) + F_2(\theta, u).
	\end{aligned}\end{equation*}
By construction, we have
	\begin{equation*}\begin{aligned}
		\sup_{\theta > 0}\Vert F_1 \Vert_{L^2} \leq & \, \sup_{\theta \in (0,2\pi/b]}\int_0^\theta \Big\|e^{i\tau H}\lambda(\eps x, y)|e^{-i\tau H}u|^{2\sigma}e^{-i\tau H}u\\
		& \, - \frac{b}{2\pi}\int_0^{\frac{2\pi}{b}}e^{i\rho H}\lambda(\varepsilon x, y)|e^{-i\rho H}u|^{2\sigma}e^{-i\rho H}u \, d\rho\Big\|_{L^2}d\tau.
	\end{aligned}\end{equation*}
Combining Corollary \ref{MoserCor} and Lemma \ref{Ubound} then gives
	\begin{equation}\begin{split}\label{F1bound}
		\Big\Vert F_1\left(\frac{t}{\varepsilon^2}, \phi(t)\right)\Big\Vert_{L^2} &\leq \sup_{\theta > 0}\Vert F_1(\theta, \phi(t)) \Vert_{L^2}\\
		&\lesssim \sup_{\theta \in (0,2\pi/b]}\int_0^\theta \Vert \lambda \Vert_{W^{2,\infty}}\Vert \phi(t) \Vert_{\Sigma^2}^{2\sigma + 1} d\tau\\
		&= \frac{2\pi}{b}\Vert \lambda \Vert_{W^{2,\infty}}\Vert \phi(t) \Vert_{\Sigma^2}^{2\sigma + 1}.
	\end{split}\end{equation}

To bound $\Vert F_2(\frac{t}{\varepsilon^2}, \phi(t))\Vert_{L^2}$, we first note that
	\begin{equation*}\begin{split}
		\Big\Vert F_2\left(\frac{t}{\varepsilon^2}, \phi(t)\right)\Big\Vert_{L^2} &\lesssim \int_0^{t/\varepsilon^2} \int_0^{\frac{2\pi}{b}}\left\Vert\Big(\lambda(\varepsilon x, y) 
		- \lambda(0, y)\Big)|e^{-i\rho H}\phi(t)|^{2\sigma}e^{-i\rho H}\phi(t)\right\Vert_{L^2}d\rho \, d\tau\\
		 &\leq \frac{t}{\varepsilon^2}\int_0^{\frac{2\pi}{b}}M^{2\sigma}\left\Vert\Big(\lambda(\varepsilon x, y) - \lambda(0, y)\Big)e^{-i\rho H}\phi(t)\right\Vert_{L^2}d\rho.
	\end{split}\end{equation*}
By the mean value theorem
\[
\left| \lambda(\varepsilon x, y) - \lambda(0, y) \right| \leq \eps |x| \|\partial_x \lambda\|_{L^\infty} ,
\]
and thus
	\begin{equation*}
		\Vert(\lambda(\varepsilon \cdot, \cdot) - \lambda(0, \cdot))e^{-i\rho H}\phi(t)\Vert_{L^2} 
		\lesssim \varepsilon\Vert \phi(t)\Vert_{\Sigma^1}\Vert \lambda\Vert_{W^{1,\infty}}.
	\end{equation*}
We therefore obtain
	\begin{equation*}
		\Big\Vert F_2\left(\frac{t}{\varepsilon^2}, \phi(t)\right)\Big\Vert_{L^2} \lesssim \frac{t}{\varepsilon}\Vert \phi(t)\Vert_{\Sigma^1}.
	\end{equation*}
which together with (\ref{F1bound}) establishes the desired estimate.
\end{proof}
A similar estimate can be shown for the second term on the r.h.s of \eqref{A2tilde}, which follows almost identically as in the proof above, 
and hence we omit the details. In summary, we then have
\[
	\left\|\int_0^t\widetilde{A_2}(s) \, ds\right\|_{L^2} \leq \eps C(T, \|\psi_0\|_{\Sigma^2}),
\]
and combining all these estimates above, we obtain
\[
	\|\phi^{\eps,\eta} (t)- \phi (t)\|_{L^2} \leq \rho(\eps) + C_1\eps + C_2\int_0^t\|\phi^{\eps,\eta} (s)- \phi (s)\|_{L^2}ds,
\]
where $C_1$ and $C_2$ are independent of $\eps$ and $\rho(\eps) \to 0$ as $\eps \to 0^+$. From here, the remainder of the proof follows analogously to the proof of item (ii) in Theorem \ref{main}.

%%%%%%%%%%%%%%%%%%%%%%%%%%%%%%%%%%%%%%%%%%%

\section{\textbf{Dynamics for polarized initial data}}\label{sec:eigen}

This section is dedicated to proving Coroallary \ref{eigenThm} for initial data of the form $\psi_0 (x,y)= \alpha_0(y)\chi_n(x)$, where $\alpha \in \mathbb C$ is some sufficiently regular amplitude. 
This follows similarly to the arguments given in \cite[Section 4]{sparber}, but for the benefit of the reader we shall present here full details: 

Denote by $P_n=| \chi_n\rangle \langle \chi_n|$ 
the projection onto the $n^{\rm th}$-eigenspace of $H$, and let 
$P_n^\perp = 1-P_n$. We shall first show that, under the dynamics of \eqref{limeqalt2}, the orthogonal complement $w (t, x,y):= P_n^\perp \phi(t,x,y)$ is identically zero 
for all $t \in [0,T_{\rm max})$. Applying $P_n^\perp$ to \eqref{limeqalt2} yields
	\begin{equation*}\begin{aligned}
		i\partial_t w &= \frac{\lambda b}{2\pi}\int_0^{2\pi/b}P_n^\perp F(\theta, \phi)d\theta\\
		&= \frac{\lambda b}{2\pi}\int_0^{2\pi/b}P_n^\perp \big(F(\theta, \phi) - F(\theta, P_n\phi)\big)d\theta + \frac{\lambda b}{2\pi}\int_0^{2\pi/b}P_n^\perp F(\theta, P_n\phi)d\theta,
	\end{aligned}\end{equation*}
where we denote $\lambda = \lambda(0,y)$ for simplicity. We can compute the second term above explicitly:
	\begin{equation*}\begin{aligned}
		\frac{\lambda b}{2\pi}\int_0^{2\pi/b}P_n^\perp F(\theta, P_n\phi)d\theta &= \frac{\lambda b}{2\pi}\int_0^{2\pi/b}P_n^\perp e^{i\theta H}(|P_n\phi|^{2\sigma}e^{-i\theta E_n}P_n\phi)d\theta\\
		&= \frac{\lambda b}{2\pi}\sum_{m \neq n}\int_0^{2\pi/b} e^{i\theta (E_m-E_n)}P_n^\perp P_m(|P_n\phi|^{2\sigma}P_n\phi)d\theta= 0.
	\end{aligned}\end{equation*}
Hence, we obtain the Duhamel formulation
	\begin{equation*}\begin{aligned}
		w(t, \cdot, \cdot) &= -\frac{i\lambda b}{2\pi}\int_0^t\int_0^{2\pi/b}P_n^\perp \big(F(\theta, \phi(s)) - F(\theta, P_n\phi(s))\big)d\theta \, ds,
	\end{aligned}\end{equation*}
where we recall that $w(0, \cdot, \cdot) = P_n^\perp \psi_0 = 0$, by assumption. We can now estimate
	\begin{equation*}\begin{aligned}
		\|w(t)\|_{L^2} &\leq \frac{b}{2\pi}\|\lambda\|_{L^\infty}\int_0^t\int_0^{2\pi/b}\big\|P_n^\perp \big(F(\theta, \phi(s)) - F(\theta, P_n\phi(s))\big)\big\|_{L^2}d\theta \, ds\\
		&\leq \frac{b}{2\pi}\|\lambda\|_{L^\infty}\int_0^t\int_0^{2\pi/b}\big\|F(\theta, \phi(s)) - F(\theta, P_n\phi(s))\big\|_{L^2}d\theta \, ds.
	\end{aligned}\end{equation*}
Next, we fix $T \in (0,T_{\rm max})$ and recall the estimate \eqref{LipEst} to find
	\begin{equation*}\begin{aligned}
		\|w(t)\|_{L^2} &\lesssim \frac{b}{2\pi}\|\lambda\|_{L^\infty}\int_0^t\int_0^{2\pi/b}\Big(\big(\|e^{-i\theta H}\phi(s)\|_{L^\infty}^{2\sigma} + \|P_n\phi(s)\|_{L^\infty}^{2\sigma}\big)\\
		&\quad\quad\quad\quad\quad\quad\quad\quad\quad\quad\quad\quad\cdot\|\phi(s) - P_n\phi(s)\|_{L^2}\Big)d\theta \, ds\\
		&\leq C(T)\int_0^t\|w(s)\|_{L^2}ds
	\end{aligned}\end{equation*}
for all $t \in [0,T]$. Here we used $\|P_n\phi(s)\|_{L^\infty} \lesssim \|P_n\phi(s)\|_{\Sigma^2} \leq \|\phi(s)\|_{\Sigma^2}$, where the last inequality holds since $P_n$ commutes with each term in the $\Sigma^2$-norm. Thus, Gr\"onwall's inequality yields
	\begin{equation*}
		\|w(t)\|_{L^2} = 0 \quad \text{for all } t \in [0,T].
	\end{equation*}
Here, $T$ was arbitrary, so we have $\phi = \alpha(t,y) \chi_n(x)$ for all $t \in [0,T_{\rm max})$.

In a second step, we shall now derive an equation for the time evolution of the amplitude $\alpha(t, \cdot)$ by plugging the ansatz $\phi (t,x,y)= \alpha(t,y) \chi_n(x)$ into \eqref{limeqalt2}. This yields
	\begin{equation*}\begin{aligned}
		i\chi_n\,  \partial_t\alpha &= \frac{\lambda b}{2\pi}\int_0^{2\pi/b}F(\theta, \alpha\chi_n)d\theta\\
		&=\frac{\lambda b}{2\pi}\int_0^{2\pi/b}e^{i\theta H}\big(|\alpha\chi_n|^{2\sigma}e^{-i\theta E_n}\alpha\chi_n\big)d\theta\\
		&=\frac{\lambda b}{2\pi}\sum_{m=0}^\infty\int_0^{2\pi/b}e^{i\theta (E_m-E_n)}P_m\big(|\alpha\chi_n|^{2\sigma}\alpha\chi_n\big)d\theta.
	\end{aligned}\end{equation*}
The integral vanishes for $m\neq n$, and thus
	\begin{equation*}\begin{aligned}
		i \chi_n\, \partial_t\alpha &=\frac{\lambda b}{2\pi}\int_0^{2\pi/b}|\alpha|^{2\sigma}\alpha P_n\big(|\chi_n|^{2\sigma}\chi_n\big)d\theta\\
		&= \lambda|\alpha|^{2\sigma}\alpha \chi_n \int_\R |\chi_n|^{2\sigma+2}dx\\
		&=  \lambda_n|\alpha|^{2\sigma}\alpha \chi_n,
	\end{aligned}\end{equation*}
where we define $\lambda_n(y) := \lambda(0,y)\|\chi_n\|_{L^{2\sigma+2}}^{2\sigma+2}\in \R$. Multiplying by $\overline \chi_n$ and using the fact that $\| \chi_n\|_{L^2}=1$ then yields
	\begin{equation*}
		i\partial_t\alpha = \lambda_n|\alpha|^{2\sigma}\alpha.
	\end{equation*}
Since $\lambda_n \in \R$ we infer from this ODE that
\[
	i\partial_t|\alpha|^2 = i\overline{\alpha}\partial_t\alpha + i\alpha\partial_t\overline{\alpha} = 0,
\]
and so $|\alpha(t,y)| =  |\alpha_0(y)|$ for all $t$. Thus, $\alpha(t, \cdot)\in \mathbb C$ must take the form
\[
	\alpha(t,y) = \alpha_0(y)e^{it\omega_n(y)}, \quad \text{for some $\omega_n(y) \in \R$,}
\]
and one easily finds $\omega_n(y) = \lambda(0,y)|\alpha_0(y)|^{2\sigma}\|\chi_n\|_{L^{2\sigma+2}}^{2\sigma+2}$, as claimed in Corollary \ref{eigenThm}.

%%%%%%%%%%%%%%%%%%%%%%%%%%%%%%%%%%%%%%%%%%%%
\appendix

\section{\textbf{Proof of Proposition \ref{NormEquiv}}}

\begin{proof} The proof is similar to one provided in \cite[Appendix A]{dele}. Recall that $H$ is the one-dimensional harmonic oscillator acting in the $x$-direction and 
let $H_\eps$ be given by \eqref{Heps}. We can then define the self-adjoint operator
	\begin{equation*}
		\mathcal T_m := \frac{1}{\varepsilon}\big(H_\varepsilon^m - H^m\big).
	\end{equation*}
We claim that there exists $C_m > 0$, such that for all $\varepsilon \in (0, 1]$,
	\begin{equation}\label{eq2}
		|\langle \mathcal T_mu, u \rangle_{L^2}| \leq C_m\Vert u\Vert_{\Sigma^m}^2,
	\end{equation}
where we denote $\langle\cdot,\cdot\rangle_{L^2}$ to be the inner product on $L^2(\mathbb{R}^2)$. Computing
	\begin{equation*}
		\Vert H_\varepsilon^{m/2}u\Vert_{L^2}^2 = \langle H^mu, u\rangle_{L^2} + \langle \varepsilon \mathcal T_mu, u\rangle_{L^2} = \Vert H^{m/2}u\Vert_{L^2}^2 + \varepsilon\langle \mathcal T_mu, u\rangle_{L^2}.
	\end{equation*}
we infer that, if \eqref{eq2} holds, then
	\begin{equation*}
		\Vert H^{m/2}u\Vert_{L^2}^2 - \varepsilon C_m\Vert u\Vert_{\Sigma^m}^2 \leq \Vert H_\varepsilon^{m/2}u\Vert_{L^2}^2 \leq \Vert H^{m/2}u\Vert_{L^2}^2 + \varepsilon C_m\Vert u\Vert_{\Sigma^m}^2.
	\end{equation*}
Therefore, we only need to show that \eqref{eq2} holds and choose $\varepsilon_m = \frac{1}{2C_m}$ to obtain the statement of the lemma for that particular value of \textit{m}. 
We will do this directly for $m=0$ and $m=1$, and then use induction for $m \geq 2$.

\smallskip

\textbf{Case $m = 0, 1$: } The estimate \eqref{eq2} holds trivially for $m = 0$. For $m = 1$, we have
	\begin{equation*}\begin{aligned}
		|\langle \mathcal T_1u, u \rangle_{L^2}| &= \big|\langle (-ibx\partial_y - \varepsilon\tfrac{1}{2}\partial^2_y)u, u \rangle_{L^2}\big|\\
		&= \big|\langle (-bx + i\varepsilon\tfrac{1}{2}\partial_y)u, (i\partial_y)u \rangle_{L^2}\big|\\
		&\leq \Vert (-bx + i\varepsilon\tfrac{1}{2}\partial_y)u\Vert_{L^2}\Vert (i\partial_y)u\Vert_{L^2}\lesssim \Vert u\Vert_{\Sigma^1}^2,
	\end{aligned}
	\end{equation*}
where in the last line we have used Cauchy-Schwarz and Lemma \ref{1}.

\smallskip

\textbf{Case $m \geq 2$:} We proceed via induction: assume \eqref{eq2} holds for $m - 1$ and $m - 2$. 
Note that this assumption implies
	\begin{equation}\label{indAs}
		\Vert H_\varepsilon^{(m-1)/2}u\Vert_{L^2}^2 \lesssim \Vert u\Vert_{\Sigma^{m-1}}^2,\quad
		\Vert H_\varepsilon^{(m-2)/2}u\Vert_{L^2}^2 \lesssim \Vert u\Vert_{\Sigma^{m-2}}^2.
	\end{equation}
We then compute
	\begin{equation*}\begin{split}
		H_\varepsilon^m &= H_\varepsilon^{m-1}(H + \varepsilon \mathcal T_1)\\
		&= (H + \varepsilon \mathcal T_1)H_\varepsilon^{m-2}H + \varepsilon H_\varepsilon^{m-1}\mathcal T_1\\
		&= H(H^{m-2} + \varepsilon \mathcal T_{m-2})H + \varepsilon \mathcal T_1H_\varepsilon^{m-2}H + \varepsilon H_\varepsilon^{m-1}\mathcal T_1\\
		&= H^m + \varepsilon H \mathcal T_{m-2}H + \varepsilon \mathcal T_1H_\varepsilon^{m-2}H + \varepsilon H_\varepsilon^{m-1}\mathcal T_1.
	\end{split}\end{equation*}
We can therefore express $\mathcal T_m$ as
	\begin{equation}\label{aa}
		\mathcal T_m =  H\mathcal T_{m-2}H +  \mathcal T_1H_\varepsilon^{m-2}H +  H_\varepsilon^{m-1}\mathcal T_1,
	\end{equation}
and bound each of the terms on the right hand side separately. First, we see that, since $H$ is self-adjoint, 
	\begin{equation}\label{a1}
		|\langle H\mathcal T_{m-2}Hu, u \rangle_{L^2}| = |\langle \mathcal T_{m-2}Hu, Hu \rangle_{L^2}|
		\leq C_{m-2}\Vert Hu\Vert_{\Sigma^{m-2}}^2
		\lesssim \Vert u\Vert_{\Sigma^{m}}^2,
	\end{equation}
where we've used (\ref{indAs}) and Lemma \ref{1}. Next, we can estimate the second term on the r.h.s. of \eqref{aa} via
	\begin{equation*}\begin{split}
		|\langle \mathcal T_1H_\varepsilon^{m-2}Hu, u \rangle_{L^2}| &= |\langle H_\varepsilon^{(m-2)/2}Hu, H_\varepsilon^{(m-2)/2}\mathcal T_1u \rangle_{L^2}|\\
		&\leq \Vert H_\varepsilon^{(m-2)/2}Hu\Vert_{L^2}\Vert H_\varepsilon^{(m-2)/2}(-ibx\partial_y - \varepsilon\tfrac{1}{2}\partial^2_y)u\Vert_{L^2},
	\end{split}\end{equation*}	
and again using (\ref{indAs}) and Lemma \ref{1} consequently yields
\begin{equation}\label{a2}
		|\langle \mathcal T_1H_\varepsilon^{m-2}Hu, u \rangle_{L^2}| 
		\lesssim \Vert Hu\Vert_{\Sigma^{m-2}}\Vert (-ibx\partial_y - \varepsilon\tfrac{1}{2}\partial^2_y)u\Vert_{\Sigma^{m-2}}
		\lesssim \Vert u\Vert_{\Sigma^{m}}^2,
\end{equation}	
Finally, we consider the last therm on the r.h.s. of \eqref{aa}:
	\begin{equation*}\begin{split}
		|\langle  H_\varepsilon^{m-1}\mathcal T_1u, u \rangle_{L^2}| &= |\langle  H_\varepsilon^{(m-1)/2}(-ibx\partial_y - \varepsilon\tfrac{1}{2}\partial^2_y)u, H_\varepsilon^{(m-1)/2}u \rangle_{L^2}|\\
		&= |\langle H_\varepsilon^{(m-1)/2}(-bx + i\varepsilon\tfrac{1}{2}\partial_y)u, H_\varepsilon^{(m-1)/2} (i\partial_y)u \rangle_{L^2}|\\
		&\leq \Vert H_\varepsilon^{(m-1)/2}(-bx + i\varepsilon\tfrac{1}{2}\partial_y)u\Vert_{L^2}\Vert H_\varepsilon^{(m-1)/2} (i\partial_y)u\Vert_{L^2}.
	\end{split}\end{equation*}
Using once more (\ref{indAs}) and Lemma \ref{1} then yields
\begin{equation}\begin{split}\label{a3}
		|\langle  H_\varepsilon^{m-1}\mathcal T_1u, u \rangle_{L^2}| 
		\lesssim \Vert (-bx + i\varepsilon\tfrac{1}{2}\partial_y)u\Vert_{\Sigma^{m-1}}\Vert (i\partial_y)u\Vert_{\Sigma^{m-1}}
		\lesssim \Vert u\Vert_{\Sigma^{m}}^2.
	\end{split}\end{equation}
Combining (\ref{a1}), (\ref{a2}), and (\ref{a3}) proves that \eqref{eq2} holds for $m\geq 2$ (assuming it does so for 
$m-1$ and $m-2$). By induction it therefore also holds for all $m \geq 0$. This completes the proof.
\end{proof}
%%%%%%%%%%%%%%%%%%%%%%%%%%%%%%%%%%%
%%%%%%%%%%%%%%%%%%%%%%%%%%%%%%%%%%%

\bibliographystyle{amsplain}

\end{document}